\documentclass[12pt]{amsart}
\usepackage{amsthm}
\usepackage{amsmath}
\usepackage{amssymb}
\usepackage{latexsym}
\usepackage{bbm}
\usepackage{enumerate}
\usepackage{amsfonts}
\usepackage{epsfig}
\usepackage{tikz}
\usetikzlibrary{fit,shapes,positioning,matrix,arrows}
\theoremstyle{plain}

\newtheorem{thm}{Theorem}[section]

\newtheorem{lemma}[thm]{Lemma}
\newtheorem{prop}[thm]{Proposition}
\newtheorem{cor}[thm]{Corollary}

\theoremstyle{definition}
\newtheorem{dfn}[thm]{Definition}
\newtheorem{eg}[thm]{Example}

\newtheorem{notation}[thm]{Notation}
\newtheorem{obs}[thm]{Observation}

\def\Z{{\mathbb Z}}

\def\O{{\mathcal O}}
\def\P{{\mathcal P}}

\def\2{{\boldsymbol 2}}

\def\q{{\mathbf q}}

\newcommand{\union}{\cup}
\newcommand{\intersect}{\cap}

\newcommand{\ov}{\overline}

\newcommand{\inv}{^{-1}}

\newcommand{\spec}{\operatorname{spec}}
\newcommand{\Aut}{\operatorname{Aut}}

\newcommand{\drank}{\delta\operatorname{-rank}}

\begin{document}

\title[Poset structure of torus-invariant prime spectra of CGL extensions]
{Poset structure of torus-invariant prime spectra of CGL extensions}

\author[Christopher Nowlin]{Christopher Nowlin}
\begin{abstract}
A key theorem of Yakimov's proves that the torus-invariant prime spectra of De Concini-Kac-Procesi algebras are isomorphic as partially ordered sets to corresponding Bruhat order intervals of Weyl groups.  We present examples of more general Cauchon-Goodearl-Letzter (CGL) extensions which exhibit this same phenomenon.  To accomplish this, we develop a procedure for iteratively constructing poset isomorphisms between torus-invariant prime spectra of CGL extensions and Bruhat order intervals of Coxeter groups.
\end{abstract}
\maketitle

\section*{Introduction}

Fix a field $k$, a positive integer $n$, and $q$ a nonzero, non-root of unity in $k$.

The algebra of quantum matrices, denoted $\O_q(M_n(k))$ has been the subject of much study as one of the primary examples of quantized coordinate rings.  In particular, a description of its prime spectrum remains a subject of interest.  The quantum matrix algebra is a $q$-skew polynomial algebra.  Loosely, this means that it resembles an algebra of polynomials over a field, but the variables do not commute; rather, they ``almost commute'' in such a way that every polynomial can be expressed uniquely as a linear combination of ordered monomials.

Moreover, $\O_q(M_n(k))$ comes equipped with a suitable action of an algebraic torus for which this algebra is a Cauchon-Goodearl-Letzter (CGL) extension (see \cite{LLR06}).  We concern ourself here with the general study of CGL extensions.

Additionally, $\O_q(M_n(k))$, along with many other important examples of CGL extensions, is  realizable as a De Concini-Kac-Procesi \cite{DKP95} algebra.  Given a word $w$ in a Weyl group $W$, one defines a De Concini-Kac-Procesi algebra as a subalgebra of the associated quantized enveloping algebra $U_q^{w,\pm}\subseteq U_q(\mathfrak g)$.  We are not here concerned with whether we look in the positive or negative quantum Borel algebra, and so use the simplified notation $U_q^w$.  All De Concini-Kac-Procesi algebras are CGL extensions (see \cite{M09}), but the reverse does not hold.  We discusss examples of such CGL extensions in this paper.

The stratification theory of Goodearl-Letzter \cite{GL00} tells us that the prime and primitive spectra of many ``quantum algebras'', including CGL extensions, can be understood through understanding the collection of prime ideals invariant under the action of an algebraic torus.  In particular, given any CGL extension, the collection of torus-invariant prime ideals is a finite partially ordered set, and we demonstrate the underlying poset is an invariant of the filtered algebra (see Propositions~\ref{prop:independentoftorus} and~\ref{prop:independt2}).

For the algebra of $\O_q(M_n(k))$, Cauchon used the deleting-derivation homomorphism to demonstrate a bijection between its invariant prime spectrum and combinatorial objects known as Cauchon diagrams \cite{C03}.  Launois gave an explicit description of the spectrum as a partially ordered set \cite{L07}.  This work was generalized by Yakimov; for any word $w\in W$, a Weyl group, Yakimov proved $\spec_HU_q^w\cong W^{\leq w}$ \cite{Y09}.  Moreover, he gave an explicit description of every torus-invariant prime ideal of $U_q^w$.

Cauchon and Meriaux take a more combinatorial approach to the problem, using Cauchon's derivation-deleting algorithm to line up torus-invariant prime ideals in De Concini-Kac-Procesi algebras with more general Cauchon diagrams which could be realized as subsets of the positive roots corresponding to the chosen word \cite{MC09}.

In this paper, we establish a combinatorial procedure for producing poset isomorphisms between torus-invariant prime spectra of CGL extensions and Bruhat order intervals of Coxeter groups.

Weyl groups are the finite Kac-Moody groups, which in turn fall under the general heading of Coxeter groups, a primary object of study in geometric group theory.  It is in the context of Coxeter groups that we most naturally see the Bruhat ordering.  Coxeter groups are realizable as groups generated by reflections of a real inner-product space.  Once a generating set has been fixed, the length of an element is defined; a reflection in a Coxeter group is a conjugate of a generator, and given an element $w$ and reflection $t$, we consider the length of $wt$, and write $w<wt$ if the length of $wt$ is greater than the length of $w$.  By closing this relation under transitivity and reflexivity, we construct a partial ordering, the Bruhat ordering, and consider Bruhat order intervals $W^{\leq w}:=[1,w]$.

Let $S=R[x;\tau,\delta]$ be a Cauchon extension, as defined in \cite{LLR06}(see Definition~\ref{dfn:CauchonExt}) and $L=S[x\inv]$.  Cauchon constructed a deleting-derivation map \cite{C03}, an isomorphism $C:R[y^{\pm1};\tau]\to L$.  This map induces a poset isomorphism $\tilde C:\spec_HR\to(\spec_HS)_{\not\ni x}$, the subset of the $H$-prime spectrum of $S$ consisting of the ideals which do not contain $x$.  We consider the following subsets of $\spec_HR$: $$(\spec_HR)_{\supseteq\delta(R)}:=\{P\in\spec_HR:\delta(R)\subseteq P\}$$ and $$(\spec_HR)_{\delta}:=\{P\in\spec_HR:\delta(P)\subseteq P\}.$$  We then see the following chain of inclusions:

$$(\spec_HR)_{\supseteq\delta(R)}\subseteq(\spec_HR)_\delta\subseteq\spec_HR\hookrightarrow\spec_HS.$$

We set $$\P_3=(\spec_HR)_{\supseteq\delta(R)},\ \P_2=(\spec_HR)_\delta\setminus\P_3,\text{ and }\P_1=\spec_HR\setminus(\P_2\cup\P_3).$$

Let $W$ be a Coxeter group and $\ov w\in W$ such that there exists an isomorphism $\nabla:\spec_HR\cong W^{\leq\ov w}$.  Suppose $a$ is a generator of $W$ such that $\ov w\leq\ov wa$.  Set $$W_a=\{w\in W:wa\leq w\}\text{ and }W_a'=\{w\in W:w\leq wa\}.$$

Set $W_3=\{w\in W^{\leq\ov w}:wa\not\leq\ov w\}$.  Consider the following chain of inclusions:

$$W_3\subseteq W_a'\cap W^{\leq\ov w}\subseteq W^{\leq\ov w}\subseteq W^{\leq\ov wa}.$$

Lining up the partially ordered sets arising from Coxeter groups and invariant prime spectra amounts to lining up the two chains of inclusions above.

Set $$W_2=W_a'\cap W^{\leq\ov w}\setminus W_3\text{ and }W_1=W_a\cap W^{\leq\ov w}.$$

We see in Theorem~\ref{thm:main} the following:

\begin{quote}
If $\nabla(W_i)=\P_i$ for $i=1,2,3$, then there exists an isomorphism $$\tilde\nabla:W^{\leq\ov wa}\to\spec_HS.$$
\end{quote}

We give an explicit construction of this isomorphism and demonstrate its compatibility with the following two projection maps.

There is a canonical contraction map $\tilde\Psi:\spec_HS\to\spec_HR$ given by $P\mapsto P\cap R$, and a projection map $\Phi:W^{\leq\ov wa}\to W^{\leq w}$ given by $$w\mapsto\begin{cases}
w,&w\leq wa\\
wa,&wa\leq w.
\end{cases}.$$

We see the following diagram commutes.

\begin{center}
\begin{tikzpicture}
\node(1) at (0,0){$W^{\leq\ov wa}$};
\node(2) at (3,0){$\spec_HS$};
\node(3) at (0,-2){$W^{\leq\ov w}$};
\node(4) at (3,-2){$\spec_HR$};
\draw[->](1)--node[auto]{$\tilde\nabla$}(2);
\draw[->](1)--node[auto]{$\Phi$}(3);
\draw[->](2)--node[auto]{$\tilde\Psi$}(4);
\draw[->](3)--node[auto]{$\nabla$}(4);
\end{tikzpicture} 

\end{center}

In Section~\ref{section:applications}, we apply Theorem~\ref{thm:main} to many classical examples of De Concini-Kac-Procesi algebras, as well as more general CGL extensions.

This work formed a portion of the author's PhD dissertation at the University of California, Santa Barbara.

\vspace{.5 cm}

\noindent{\bf Acknowledgements.}
The author would like to acknowledge Ken Goodearl for helpful suggestions.

\section{Background and notation}

Fix a field $k$ and $q\in k^\times$.  Set $\hat q=q-q\inv$.  We use $n$ to denote a positive integer.

A \emph{Coxeter group} is a group with presentation $\langle s_1,\dots,s_n:(s_is_j)^{m_{ij}}=1\rangle$ where $m_{ii}=1$ for all $i$ and $2\leq m_{ij}=m_{ji}$ for all $i\neq j$.  We refer to the matrix $m=(m_{ij})$ as the \emph{Coxeter matrix}.  We will abuse notation and refer to a Coxeter group as the group with respect to a fixed generating set.

There is a standard ordering on Coxeter groups known as the Bruhat order, which we will denote $\leq$.  We recall some key properties.  The \emph{chain property} tells us that Bruhat posets are graded by length.  The \emph{subword property} tells us that $w'<w$ in the Bruhat order precisely when $w'$ can be expressed as a (not necessarily contiguous) subword of $w$.  The \emph{lifting property} tells us the following: Let $a$ be a generator for $W$ and $w< w'\in W$ with $w<wa$ and $w'a<w'$; then $w\leq w'a$ and $wa\leq w'$.  We will need many technical lemmas regarding Bruhat order intervals, most of which are straightforward applications of the above properties.  See for example \cite{BB05} for details.

Weyl groups are a special subset of Coxeter groups indexed by Dynkin diagrams.  We recall and fix a presentation of the groups of types $A_n$ and $D_{n+1}$ here, as they will be useful in this paper.  The Coxeter matrix is encoded in the Dynkin diagram by letting $m_{ij}=2$ when the nodes corresponding to $s_i$ and $s_j$ are not connected by an edge and letting $m_{ij}=3$ when they are.  The Coxeter group of type $A_n$ will have generating set denoted $s_1,\dots,s_n$ and relations encoded by the Dynkin diagram:\begin{center}
\begin{tikzpicture}
[vertex/.style={circle,draw,inner sep=0 pt, minimum size=5}]
\node(1) [vertex]{};
\node(2) [vertex] [right=of 1]{};
\node(3) [vertex] [right=of 2]{};
\node(n-1) [vertex][right=1.5 cm of 3]{};
\node(n) [vertex][right=of n-1]{};
\draw(1)--(2);
\draw(2)--(3);
\draw(n-1)--(n);
\draw[loosely dashed] (3)--(n-1);
\end{tikzpicture}
\end{center}  The Coxeter group of type $D_{n+1}$ will have generating set $s_1,\dots,s_{n+1}$ and relations encoded by:
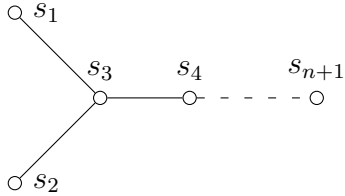
\begin{center}
\begin{tikzpicture}
[vertex/.style={circle,draw,inner sep=0 pt, minimum size=5}]
\node(2) [vertex][label=above:$s_3$]{};
\node(0) [vertex] [above left=of 2,label=right:$s_1$]{};
\node(1) [vertex] [below left=of 2,label=right:$s_2$]{};
\node(3) [vertex][right=of 2,label=above:$s_4$]{};
\node(n) [vertex][right= 1.5 cm of 3,label=above:$s_{n+1}$]{};
\
\draw(0)--(2);
\draw(1)--(2);
\draw(2)--(3);
\draw[loosely dashed] (3)--(n);
\end{tikzpicture} 
\end{center}

We assume familiarity with noetherian ring theory and quantum algebras as can be found in \cite{GW89} and Chapters I and II of \cite{BG02}.  We summarize briefly important later results and definitions extending the work summarized in \cite{BG02}, specifically from \cite{C03} and \cite{LLR06}.

\begin{dfn}Given a noetherian algebra $R$, we say an Ore extension $S=R[x;\tau,\delta]$ is a \emph{Cauchon extension} of $R$ if the following are satisfied:
\begin{enumerate}[(i)]
\item $\delta$ is locally nilpotent.
\item $\delta\tau=q\tau\delta$ with $q$ not a root of unity.
\item There is an abelian group $H$ acting on $S$
by algebra automorphisms.
\item $R$ is $H$-stable and $x$ is an $H$-eigenvector.
\item There exists $h_0\in H$ such that $h_0|_R=\tau$ and the
$h_0$-eigenvalue of $x$ is not a root of unity.
\item Every $H$-prime ideal of $R$ is completely prime.
\end{enumerate}\label{dfn:CauchonExt}
\end{dfn}

\begin{dfn} A \emph{Cauchon-Goodearl-Letzter (CGL) extension} is an iterated Ore extension of the form
$S=k[x_1][x_2;\tau_2,\delta_2]\cdots[x_n;\tau_n,\delta_n]$ satisfying the following conditions, where $R_1=k[x_1]$, and $R_j=R_{j-1}[x_j;\tau_j,\delta_j]$:\begin{enumerate}[(a)]
\item Each $\delta_j$ is locally nilpotent and there exist nonroots of unity $q_j\in k^\times$ with $\delta_j\tau_j=q_j\tau_j\delta_j$.
\item There is a torus $H$ acting rationally on $S$ by algebra automorphisms.
\item For $i=1\dots n$, $x_i$ is an $H$-eigenvector.
\item There exist elements $h_1,\dots,h_n\in H$ such that $h_j(x_i)=\tau_j(x_i)$ for $i<j$ and the $h_j$-eigenvalue of $x_j$ is not a root of unity.
\end{enumerate}

\label{dfn:CGL}
\end{dfn}

\begin{eg}Given a word $w$ in a Weyl group $W$, there exists a De Concini-Kac-Procesi algebra denoted $U_q^w$, a subalgebra of the quantized enveloping algebra corresponding to $W$.  Every $U_q^w$ is a CGL-extension.\cite{M09}

\end{eg}

\begin{thm}[\cite{C03}]Let $S=R[x;\tau,\delta]$ be a Cauchon extension.  Then
there exist nonzero scalars $q_0,q_1,\dots, $ given by $$q_l=q^{l(l+1)/2}\frac{(q-1)^{-l}}{(\ov l)!_q},$$ so that there exists an algebra homomorphism $C:R\to S[x\inv]$ given by
$$r\mapsto\sum_{l=0}^\infty q_l\delta^l\tau^{-l}(r)x^{-l},$$ which extends to an isomorphism $R[y^{\pm 1};\tau]\to
S[x\inv]$ where $y\mapsto x.$

\label{thm:Cauchon}\end{thm}

\begin{thm}[\cite{Y09}]Let $W$ be a Weyl group.  For $w\in W$, $\spec_H(U_q^w)\cong W^{\leq w}$
as partially ordered sets.\label{thm:Yaki}
\end{thm}

We conclude this section by fixing some standard poset notation we will need

Given disjoint partially ordered sets $\P_1$ and $\P_2$, we denote by $\P=\P_1\sqcup\P_2$ the disjoint union of posets, with poset structure given by the rule: $$x\leq_\P x'\Leftrightarrow\exists i,\  x,x'\in\P_i\text{ and }x\leq_{\P_i}x'.$$

Given partially ordered sets $\P_1$ and $\P_2$, we denote by $\P=\P_1\times\P_2$ the cartesian product of posets with ordering given by $$(x,y)\leq_{\P}(x',y')\Leftrightarrow x\leq_{\P_1}x'\text{ and }x\leq_{\P_2}x'.$$

\section{Bruhat order intervals}\label{section:bruhat}

\begin{notation}
For this section, $W$ refers to a Coxeter group with respect to a fixed generating set.  Bruhat order intervals refer to intervals in $W$ of the form $[1,w]$.
\end{notation}

In this section, we give some combinatorial properties of posets and Coxeter groups in the hopes of employing a combinatorial approach to the problem of finding an isomorphism of posets between Bruhat order intervals and torus-invariant prime spectra.

\begin{notation}Let $W$ be a Coxeter group and $a$ a generator.  Denote by $m_a:W\to W$ the right multiplication map $w\mapsto wa.$  Note that $m_a$ is a bijection but does not preserve Bruhat order.  Denote by $W_a=\{w\in W:wa<w )\}$ and $W_a'=W\setminus W_a.$\label{notation:Wa}

\end{notation}

We next collect some standard equivalences it will be useful to have handy.

\begin{obs}The following statements are equivalent for $w\in W$ and $a$ a generator of $W$.
\begin{enumerate}[(a)]
\item $wa\leq w$
\item $wa<w$
\item $l(wa)<l(w)$
\item $l(wa)\leq l(w)$
\item There exists $w'\in W$ with $w'\leq w$ such that $w=w'a$.
\end{enumerate}
\label{prop:DRw}
\end{obs}

\begin{obs}The restrictions $m_a:W_a\to W_a'$ and $m_a:W_a'\to W_a$ are mutually inverse isomorphisms of partially ordered sets.\label{prop:ma}\end{obs}

\begin{notation}Let $a$ be a generator of $W$.  Let $\ov w\in W_a'$.  We fix this notation for the remainder of the article.

We define four pairwise disjoint subsets $W_i^{\ov w,a}$ for $i=1,\dots,4$ of $W^{\leq\ov wa}$ as follows.\begin{align*}
W_1^{\ov w,a}&=\{w\in W^{\leq\ov wa}:wa\leq w\leq\ov w\}\\
W_2^{\ov w,a}&=\{w\in W^{\leq\ov wa}:w\leq wa\leq \ov w\}\\
W_3^{\ov w,a}&=\{w\in W^{\leq\ov wa}:wa\not\leq\ov w\}\\
W_4^{\ov w,a}&=\{w\in W^{\leq\ov wa}:w\not\leq\ov w\}
\end{align*}

For ease of notation, we will often write $W_i:=W_i^{\ov w,a}$ for $i=1,\dots,4$ in cases where $\ov w$ and $a$ are fixed and understood.

We will refer to elements of $W_4$ as \emph{new words} and all others as \emph{old words}.
\label{notation:BruhatPartition}
\end{notation}

\begin{obs}Notice that $W_1\cup W_2\cup W_3=W^{\leq\ov w}$.

Further, observe $$W_1=W_a\cap W^{\leq\ov w},\
W_2=m_a(W_1),\ W_3=m_a(W_4)$$

Also notice \begin{align*}
W_2\cup W_3&=W^{\leq\ov wa}\cap W_a'=W^{\leq\ov w}\cap W_a'.\\
W_1\cup W_4&=W^{\leq\ov wa}\cap W_a.
\end{align*}

Finally, observe that $W_3$ is an upper set of $W^{\leq\ov w}$, and that $W_4$ and $W_3\cap W_4$ are both upper sets of $W^{\leq\ov wa}.$\label{obs:BruhatPartitionfacts}

\end{obs}

\begin{obs}There is a bijective poset homomorphism $\psi:W_2\times\2\to W_1\cup W_2$ $$(w,\epsilon)\mapsto\begin{cases}
w,&\epsilon=0\\
wa,&\epsilon=1.
\end{cases}$$ and canonical inclusions $W_2\to W_1\cup W_2$ and $W_2\to W_2\times\2$, $w\mapsto(w,0)$ so that the following diagram commutes:\begin{center}\begin{tikzpicture}
\matrix (m)[matrix of math nodes, row sep=1.5cm, column sep=1.5cm]{W_2\times\2&&W_1\cup W_2\\
&W_2&\\};
\path[->](m-1-1)edge node[auto]{$\psi$}(m-1-3);
\path[left hook->](m-2-2)edge(m-1-1);
\path[right hook->](m-2-2)edge(m-1-3);
\end{tikzpicture}

\end{center}\label{lemma:W1W2toW2times2}\end{obs}

\begin{prop}There is a bijective homomorphism $\nu_1:W_3\sqcup(W_2\times\2)\to W^{\leq\ov w}$ given by $$\xi\mapsto\begin{cases}
\xi,&\xi\in W_3\\
\psi(\xi),&\xi\in W_2\times\2\end{cases}$$ with $\psi$ as in Observation~\ref{lemma:W1W2toW2times2}.  There is a canonical injective poset homomorphism $$\nu_2:W_3\sqcup(W_2\times\2)\to (W_3\cup W_2)\times\2$$ given by $$\xi\mapsto\begin{cases}\xi,& x\in (W_2\times \2)\\
(\xi,0),&\xi\in W_3.
\end{cases}$$

There is a canonical inclusion $W^{\leq\ov w}\to W^{\leq\ov wa}$ and a bijective homomorphism $\intercal:(W_3\cup W_2)\times\2\to W^{\leq\ov wa}$ given by $$(w,\epsilon)\mapsto\begin{cases}w,&\epsilon=0\\
wa,&\epsilon=1.\end{cases}$$

Then the following commutative diagram is a pushout in the category of finite posets.\begin{center}\begin{tikzpicture}
\node(1) {$(W_2\cup W_3)\times\2$};
\node(2) [right=1.5cm of 1]{$W^{\leq\ov wa}$};
\node(3) [below=of 1]{$W_3\sqcup(W_2\times\2)$};
\node(4) [right=1.5cm of 3]{$W^{\leq\ov w}$};
\draw[->](1)--node[auto]{$\intercal$}(2);
\draw[<-](1)--node[auto]{$\nu_2$}(3);
\draw[right hook->](4)--(2);
\draw[->](3)--node[auto]{$\nu_1$}(4);
\end{tikzpicture} 
\end{center}
\label{prop:BruhatPushout}
\end{prop}

\begin{proof}

That $\nu_1$ is a well-defined bijective homomorphism is clear.  Notice that $\nu_1$ is not an isomorphism.  It is similarly clear that $\nu_2$ is a well-defined injective homomorphism and that $\intercal\nu_2=\nu_1.$

We observe that $\intercal\inv:W^{\leq\ov wa}\to(W_3\cup W_2)\times\2$ is given by $$w\mapsto\begin{cases}(w,0),&w\in W_a'\\
(wa,1),&w\in W_a.\end{cases}$$  We  note that $\intercal$ is order-preserving but $\intercal\inv$ is  not.

It remains to verify the universality condition.  Suppose $(\tilde\P,\theta_1,\theta_2)$ is the pushout.  We then have a unique order-preserving map $f:\tilde\P\to W^{\leq\ov wa}$ satisfying\begin{center}\begin{tikzpicture}
\node(1) at (0,0){$(W_2\cup W_3)\times\2$};
\node(2) at (3,0){$\tilde P$};
\node(3) at (0,-2){$W_3\sqcup(W_2\times\2)$};
\node(4) at (3,-2){$W^{\leq\ov w}$};
\node(W) [above right=of 2]{$W^{\leq\ov wa}$};
\draw[->](1)--node[auto]{$\theta_2$}(2);
\draw[<-](1)--node[auto]{$\nu_2$}(3);
\draw[->](4)--node[auto]{$\theta_1$}(2);
\draw[->](3)--node[auto]{$\nu_1$}(4);
\draw[right hook->](4)--(W);
\draw[->](1)--node[auto]{$\intercal$}(W);
\draw[dashed,->](2)--node[left]{$f$}(W);
\end{tikzpicture} \end{center}

Notice also that we have that $\theta_2$ is bijective.

We wish to show $f$ is an isomorphism.  We define $\ov f:=\theta_2\intercal\inv:W^{\leq\ov wa}\to\tilde\P$ and check that for all $w\in W^{\leq\ov wa}$, $$f\ov f(w)=f\theta_2\intercal\inv(w)=\intercal\intercal\inv(w)=w.$$  Recalling these are finite posets, we conclude $f$ is a bijection with inverse $\ov f$.

It remains to check that $\ov f$ preserves order.  Notice that $\intercal\inv$ restricts to order-preserving maps on $W^{\leq\ov wa}\cap W_a$ and $W^{\leq\ov wa}\cap W_a'$ respectively.

Assume $w,w'\in W^{\leq\ov wa}$ with $w\leq w'$.  We consider the two nontrivial cases.

\begin{enumerate}[{Case} 1:]
\item $w\in W_a$ and $w'\in W_a'$.

Since $W_4$ is an upper set, we must have $w\in W_1$.  Hence, $\intercal\inv(w)=(wa,1)$, with $wa\in W_2$.

We thus see that \begin{align*}
\theta_2\intercal\inv(w)&=\theta_2(wa,1)\\
&=\theta_2\nu_2(wa,1)&\text{ as }(wa,1)\in W_2\times2\\
&=\theta_1\nu_1(wa,1)\\
&=\theta_1(w).
\end{align*}

Notice $\intercal\inv(w')=(w',0)$.  Further observe that if $\nu_2(\xi)=(w',0)$, then $\nu_1(\xi)=w'.$  Hence, \begin{align*}
\theta_2\intercal\inv(w')&=\theta_2(w',0)\\
&=\theta_2\nu_2(\xi)&\text{ for some }\xi\in W_3\sqcup(W_2\times\2)
&=\theta_1\nu_1(\xi)\\
&=\theta_1(w').
\end{align*}

We thus see that $\ov f(w)=\theta_1(w)\leq\theta_1(w')=\ov f(w')$, as desired.

\item $w\in W_a'$ and $w'\in W_a$.  By the lifting property, we see $w\leq w'a$.  Hence, $\intercal\inv(w)=(w,0)\leq (w'a,1)=\intercal\inv(w')$.  Thus $\ov f(w)\leq\ov f(w').$

\end{enumerate}

We conclude $W^{\leq\ov wa}\cong\tilde\P$, and so the desired pushout.\end{proof}

The following is an easy corollary.

\begin{cor}
Suppose $W$ is a Coxeter group, $a$ a generator, and $\ov w\in W_a'$ with $a\not\leq\ov w$.  Then $W^{\ov wa}\cong W^{\leq\ov w}\times\2\cong W^{\leq\ov w}\times W^{\leq a}$.\label{cor:times2}
\end{cor}

\begin{dfn}
We say a word $\ov w$ is \emph{decomposable} if there exist $w,w'\in W$ with $\ov w=ww'$, such that, given any $z\in w$, $$z\leq w\text{ and }z\leq w'\Rightarrow z=1.$$  We say a word is \emph{indecomposable} otherwise.
\end{dfn}

\begin{cor}Suppose $\ov w\in W$ is a decomposable word which can be decomposed $\ov w=ww'$.  Then $W^{\leq\ov w}\cong W^{\leq w}\times W^{\leq w'}.$\end{cor}

It follows that understanding the poset structure of Bruhat order intervals reduces to understanding the Bruhat order intervals of indecomposable words.

\begin{lemma}
If $w\in W_2$, $w'\in W_4$ and $w\leq w'$, then there exists $z\in W_3$, $w\leq z\leq w'$.\label{obs:ignoreW1}
\end{lemma}

\begin{proof}Set $z=w'a$.\end{proof}

We wish to understand the poset structure of $W^{\leq\ov wa}$ in terms of $W^{\leq\ov w}$.  Such an understanding will pave the way for proving facts about Bruhat order intervals by induction.    We proceed as though the poset structure of $W^{\leq\ov w}$ is fully understood.

Since $W^{\leq\ov w}\subseteq W^{\leq\ov wa}$, all order relations involving old words are already known to us.  Recall we refer to $W_4$ as the set of new words and that $m_a$ maps $W_3$ isomorphically onto $W_4$.  Thus all relations involving only new words are also understood.

We are left to wonder about relations between old words and new words.  Since $W_4$ is an upper set, half of this problem is trivialized.  Given $w\in W_4$, what old words $w'$ satisfy $w'\leq w$?  Given that we have a graded poset, transitivity reduces this to the question: What old words $w'$ with $l(w')=l(w)-1$ satisfy $w'\leq w$?

By Lemma~\ref{obs:ignoreW1}, we can ignore words in $W_2$.  After these reductions, we have the next few observations.

\begin{obs}
The map $\intercal$ of Proposition~\ref{prop:BruhatPushout} restricts to an isomorphism $W_3\times\2\cong W_3\cup W_4$.
\end{obs}

\begin{notation}Define $\Phi:W^{\leq\ov wa}\to W_2\cup W_3$ by $$w\mapsto\begin{cases}w,& w\in W_a'\\
wa,& w\in W_a.\end{cases}$$\label{notation:Phi}\end{notation}

\begin{obs}
The map $\Phi$ is a 2-1 surjective poset homomorphism.  Moreover, $\Phi$ restricts to the identity map on $W_a'\cap W^{\leq\ov wa}=W_1\cup W_3$ and to an isomorphism $$W_a\cap W^{\leq\ov wa}\to W_a'\cap W^{\leq\ov wa}.$$\label{lemma:Phi1}
\end{obs}

\begin{obs}Suppose $w\in W_a$ and $w'\in W_a'$.  Then $$w'\leq w\Leftrightarrow\Phi(w')\leq\Phi(w).$$\label{lemma:Phi2}\end{obs}

We remark that $\Phi$ can be viewed as a projection on $W^{\leq\ov wa}$ in that $\Phi^2=\Phi$.

Now suppose $w\in W_4$.  From Observations~\ref{lemma:Phi1} and \ref{lemma:Phi2}, we have now that $w'\leq w$ if and only if $\Phi(w')\leq\Phi(w)$. Further notice that $\Phi(w)\leq w$.

\begin{notation}Let $\tilde \P$ be a partially ordered set with $\tilde \P=\P\cup \P_x$ such that $\P\cap\P_x=\emptyset$.  Suppose that $\P_x$ is an upper set of $\tilde \P$.  Suppose there exists a poset homomorphism $\tilde\Phi:\tilde \P\to \P$ such that the restriction $\tilde\Phi:\P_x\to\tilde\Phi(\P_x)$ is an isomorphism.  Moreover, suppose $\tilde\Phi(p)\leq p$ for all $p\in \P_x$.  Finally, suppose that for all $p'\in \P_x$ and $p\in \P$, $$p\leq p'\Leftrightarrow\tilde\Phi(p)\leq\tilde\Phi(p').$$\label{notation:poset_setup}\end{notation}

\begin{prop}Use the notation of \ref{notation:poset_setup}.  Suppose there exists a poset isomorphism $\nabla:W^{\leq\ov w}\to \P$ such that
\begin{enumerate}[(a)]
\item $\nabla(W_3)=\tilde\Phi(\P_x)$.
\item $\tilde\Phi\nabla=\nabla\Phi$ on $W^{\leq\ov w}$.
\end{enumerate}

Then $\nabla$ extends to an isomorphism $\widetilde\nabla:W^{\leq\ov wa}\to \tilde \P$ given by $$\widetilde\nabla(w)=\begin{cases}\nabla(w),&w\in W^{\leq\ov w}\\
(\tilde\Phi|_{\P_x})\inv\nabla(wa),&\text{else}.
\end{cases}$$\label{prop:IsoExtends}
\end{prop}

\begin{proof}For $w\in W^{\leq\ov wa}$, notice that $w\not\in W^{\leq\ov w}$ is equivalent to saying $w\in W_4$.  It follows that $wa\in W_3$.  Hence, by (a), $\nabla(wa)\in\tilde\Phi(\P_x)$, as needed.

We note $\tilde\nabla$ is a bijection with inverse $$\tilde\nabla\inv:p\mapsto\begin{cases}\nabla\inv(p),&p\in \P\\
m_a\nabla\inv\tilde\Phi(p),&p\in \P_x.
\end{cases}$$

Notice next that for $w\in W_4$, $\tilde\Phi\tilde\nabla(w)=\nabla\Phi(w).$  Hence, $\tilde\Phi\tilde\nabla=\nabla\Phi$ on all of $W^{\leq\ov wa}.$

It is further clear that $\tilde\nabla$ restricts to isomorphisms on both $W^{\leq\ov w}$ and $W_4$.  It remains to check that $\tilde\nabla$ preserves orderings between old and new words.  Since $W_4$ is an upper set, we need only consider $w\in W^{\leq\ov w}$ and $w'\in W_4$ with $w\leq w'$ to ensure $\tilde\nabla$ is order-preserving.

We see \begin{align*}
w\leq w'&\Rightarrow\nabla\Phi(w)\leq\nabla\Phi(w')\\
&\Rightarrow\tilde\Phi\tilde\nabla(w)\leq\tilde\Phi\tilde\nabla(w')\\
&\Rightarrow\tilde\nabla(w)\leq\tilde\nabla(w')\end{align*} since $\tilde\nabla(w')\in\P_x$.

It remains to check that $\tilde\nabla\inv$ preserves order.  We assume $p\in\P$ and $p'\in\P_x$ with $p\leq p'$.  Recalling Lemmas~\ref{lemma:Phi1} and ~\ref{lemma:Phi2}, we see \begin{align*}
p\leq p'&\Rightarrow\nabla\inv\tilde\Phi(p)\leq\nabla\inv\tilde\Phi(p')\\
&\Rightarrow\Phi\tilde\nabla\inv(p)\leq\Phi\tilde\nabla\inv(p')\\
&\Rightarrow \tilde\nabla\inv(p)\leq\tilde\nabla\inv(p')
\end{align*} since $\tilde\nabla\inv(p')\in W_4$.
\end{proof}

Recall that $\Phi$ restricts to the identity map on $W_2\cup W_3$.  We then notice the following simplification.

\begin{prop}In the case where $\tilde\Phi$ is a projection, hypothesis (b) of Proposition~\ref{prop:IsoExtends} can be replaced with the hypothesis that $\tilde\Phi\nabla=\nabla\Phi$ on $W_1\cup W_2$, and the conclusion still holds.\label{remark:ProjSuffices}\end{prop}

\section{Torus-invariant prime spectra of Cauchon extensions}\label{section:CGLPrelims}

For this section, let $S=R[x;\tau,\delta]$ be a Cauchon extension with respect to a torus $H$.  In particular, $q$ is not a root of unity and so $k$ is an infinite field.  There exists $h_0\in H$ satisfying condition (v) of Definition~\ref{dfn:CauchonExt}

\begin{notation}

Set $L=S[x\inv]$

We will refer to the map $C:R\to L$ from Theorem~\ref{thm:Cauchon} as the
\emph{derivation-deleting homomorphism} or \emph{Cauchon's homomorphism}.  The extension of $C$ to $R[y^{\pm1};\tau]$ will be referred to analogously as an isomorphism.

Denote by $\iota:\spec_HR\to\{\text{ ideals of }S\}$ the canonical extension map $P\mapsto\langle P\rangle_S$.

Given an ideal $I$ of $S$ with $x\not\in I$, set $I_x=\langle I\cup\{x\}\rangle_S$.

Given $P\in\spec_HR$, denote $\langle P\cup\{x\}\rangle_S$ by $P_x$.

This notation will hold for the rest of the section.

\end{notation}

Notice that elements of $S$ can be expressed uniquely $$r_0+r_1x+\cdots r_nx^m$$ for some $m$ with each $r_i\in R$.  For an ideal $P$ of $R$, we denote the right ideal of $S$ generated by $P$ as $PS$ and note elements of $PS$ have the form $r_0+r_1x+\cdots r_mx^m$ for some $m$, with each $r_i\in P$, $i=0,\dots,m$.

If $Q\in\spec S$, then the extension of $Q$ to $L$ is given by $$Q^e=\{rx^{-m}:r\in Q,\ m\in\Z_{\geq0}\}.$$  For $I\in\spec L$, the contraction of $I$ to $S$ is given by $$I^c:=I\cap S=\{rx^m:r\in I,\ m\in\Z_{\geq0}\}.$$


We consider finitely generated positively
filtered algebras generated by the first component of the filtration.  That is, we consider an algebra $A=\union_{l\geq 0}A_l$, where $A$ is generated by $A_1$, and each $A_l$ is a finite-dimensional $k$-subspace of $A$ such that $$A_0\subseteq A_1\subseteq A_2\subseteq\cdots.$$

Further, we have the filtration property, that for $a\in A_s$ and $b\in A_t$, $ab\in A_{s+t}$.

In our examples, we will be interested in pairs $(A,V)$ where $A$ is an algebra and $V$ a finite
dimensional vector space which generates $A$ as an algebra such that $1\in V$. Notice that $A$ is filtered by $V^l$ for $l\geq 0$.\label{remark:CatFiltered}

\begin{prop}Let $A$ be a finitely generated positively filtered algebra as above and consider the
algebraic group $\Aut_{fil}(A)$ of filtered algebra automorphisms
of $A$. Choose a maximal torus $H\subseteq\Aut_{fil}(A)$.

Then the partially ordered set $\spec_H(A)$ is independent of the
choice of maximal torus $H$.
\label{prop:independentoftorus}
\end{prop}

\begin{proof} Let $A$ be a such a filtered algebra.  Then there is a
natural embedding $\Aut_{fil}(A)$ into the general linear group of
$A_1$, and $\Aut_{fil}A$ can thus be realized as a closed subvariety of $GL(A_1)$ and so an algebraic group.

Suppose $H_1,H_2$ are two maximal tori in $\Aut_{fil}(A)$.  Then
$H_2$ is a conjugate of $H_1$ (\cite[Theorem 10.6]{B69}).  Let $g\in\Aut_{fil}(A)$ so that $gH_1g\inv=H_2$.

Then there is a poset isomorphism $\spec_{H_1}A\to\spec_{H_2}A$ given by $P\mapsto gP$.
\end{proof}

\begin{prop}Assume $k$ is infinite with $A$ as above.  Suppose $H,H'\subseteq\Aut_{fil}(A)$ are tori such that $\spec_HA$ and $\spec_{H'}(A)$ are finite.  Then $\spec_H(A)\cong\spec_{H'}(A)$.\label{prop:independt2}
\end{prop}

\begin{proof}Every torus is contained within a maximal torus.  Suppose $H\subseteq\tilde H$ with $\tilde H$ a maximal torus.  We seek to show that $\spec_H(A)\cong\spec_{\tilde H}(A)$.  The same will then hold for $H'$, at which point the result will follow from Proposition~\ref{prop:independentoftorus}.

Suppose $P\in\spec_HA$.

Notice it follows that $\tilde h(P)\in\spec_HA$ for all $\tilde h\in\tilde H$.  Indeed, for all $h\in H$, $$h\tilde h(P)=\tilde hh(P)=\tilde h(P).$$

Notice that $H\subseteq\tilde H$ implies $\spec_{\tilde H}(A)\subseteq\spec_H(A)$.

Thus the $\tilde H$-orbit of $P$ is finite.  But then the orbit of $\tilde H$ is a singleton by \cite[II.2.9]{BG02}.  Hence, $P\in\spec_{\tilde H}A$.  We thus see $\spec_HA=\spec_{\tilde H}A$.
\end{proof}

Thus we have an isomorphism invariant of finitely generated positively filtered algebras in
the form of a partially ordered set.

Any iterated skew polynomial algebra (e.g. any CGL extension) is a filtered algebra with generating vector space the span of the generators.

\begin{obs}
The algebra $R[y;\tau]$ is also a Cauchon extension with respect to $H$.
\end{obs}

\begin{notation} We fix some useful subsets of torus-invariant prime spectra here.\begin{align*}
(\spec_HS)_{\ni x}&:=\{Q\in\spec_HS:x\in Q\};\\
 (\spec_HS)_{\not\ni x}&:=\{Q\in\spec_HS:x\not\in Q\};\\
 (\spec_HR)_\delta&:=\{P\in\spec_HR:\delta(P)\subseteq P\};\\
 (\spec_HR)_{\supseteq\delta(R)}&:=\{P\in\spec_HR:\delta(R)\subseteq P\}.
 \end{align*}

\end{notation}

\begin{dfn}Let $A$ be an algebra, and $\delta$ a locally nilpotent skew derivation of $A$.  For $0\neq r\in A$, the \emph{$\delta$-rank} of $r$ is the largest integer $k$ so that $\delta^k(r)\neq 0$.

We set $\drank(0):=-1$.\end{dfn}

\begin{prop}There is a poset isomorphism $\tilde{C}:\spec_H(R)\to(\spec_HS)_{\not\ni x}$ given by $$P\mapsto \langle C(P)\rangle_L\intersect
S.$$\label{prop:Cinduced}\end{prop}

\begin{proof}
We first recall $$\spec_H(R)\cong\spec_HR[y^{\pm1};\tau]$$ via
extension of ideals.
As $C$ extends to an isomorphism of algebras $R[y^{\pm1}:\tau]\to L$, $C$ induces an isomorphism
of partially ordered sets $\spec_HR\to\spec_HL$.

Recall also that the set
$\{1,x,x^2\dots\}$ is a right denominator set in $S$, whence $\spec_H L\cong(\spec_HS)_{\not\ni x}$ via $Q\mapsto
Q\intersect S$, since all prime ideals of $L$ are completely prime.

Further, notice that $J\in\spec L$ is $H$-stable if and only if
$J\intersect S$ is $H$-stable.  The result follows.
\end{proof}

\begin{notation}We fix the notation of $\tilde C$ from Proposition~\ref{prop:Cinduced}.
\end{notation}

\begin{obs}For $Q\in\spec_HS$, if $x\in Q$ then $\delta(R)\subseteq Q$.\end{obs}

\begin{lemma}Let $P\in\spec_HR$.  The following are equivalent.
\begin{enumerate}[(a)]
\item $\delta(P)\subseteq P$.
\item $PS=SP$.
\item $\iota(P)=PS$.
\item $\iota(P)\cap R=P$.
\end{enumerate}

\label{lemma:deltaequiv}
\end{lemma}

\begin{proof}
Assume (a).  Notice $SP\subseteq PS+\delta(P)=PS$.  Analogously, $PS\subseteq SP$.  This gives us (b).  Then $PS$ is a two-sided ideal with $P\subseteq PS\subseteq\iota(P)$.  Hence, $PS=\iota(P)$, verifying (c).  (d) follows trivially.

Assume (d).  Notice $\delta(P)\subseteq\iota(P)$.  (a) follows.  We thus see (a)-(d) are equivalent.\end{proof}

Recalling that $\tau$ is given by an element of $H$, the following is a well-known result.

\begin{lemma}For $P\in(\spec_HR)_\delta$, $$(R/P)[x';\tau',\delta']\cong S/\iota(P),$$ where $\tau'$ (resp. $\delta'$) is the automorphism (resp. skew derivation) on $R/P$ induced by $\tau$ (resp. $\delta$).\label{lemma:CGL1}\end{lemma}

\begin{lemma}If $P\in(\spec_HR)_\delta$, then $\iota(P)\in\spec_HS$.

Moreover, $(\spec_HR)_\delta\cong\iota\left((\spec_HR)_\delta\right)$ as partially ordered sets via $\iota$.  The inverse map is the contraction map, $Q\mapsto Q\cap R$.
\label{lemma:iota}\end{lemma}

\begin{proof}We have that $P$ is completely prime in $R$, so $R/P$ is a domain.  Since $\delta(P)\subseteq P$, we see from Lemma~\ref{lemma:CGL1} that $S/\iota(P)\cong(R/P)[x;\tau',\delta']$ is also a domain, and so $\iota(P)$ is completely prime.  Recalling from Lemma~\ref{lemma:deltaequiv} that $\iota(P)=PS$, we see that $\iota(P)$ is also $H$-stable, so $\iota(P)\in\spec_HS$.

Also from Lemma~\ref{lemma:deltaequiv}, we see $\iota(P)\cap R=P$.
\end{proof}

\begin{lemma}

The poset $$(\spec_HR)_{\supseteq\delta(R)}\cong(\spec_HS)_{\ni x}$$  via a poset isomorphism $P\mapsto P_x.$

The inverse map is the contraction map, $Q\mapsto Q\cap R.$
\label{lemma:CGL3}
\end{lemma}

\begin{proof}

Consider $P\in(\spec_HR)_{\supseteq\delta(R)}.$  Set $P'=\iota(P)$.  We have from Lemma~\ref{lemma:iota} that $P'\in\spec_HS$.  Moreover, we see from Lemma~\ref{lemma:CGL1} that $$S/P'\cong (R/P)[x';\tau',\delta']=(R/P)[x';\tau'].$$  It follows that $S/P_x\cong R/P$, and hence $P_x$ is completely prime and so an element of $\spec_HS$.  This verifies the map $P\mapsto P_x$ is well-defined.

That $P\mapsto P_x$ is order-preserving is clear.

Notice that $P_x$ is the set of polynomials in $x$ with constant term in $P$.

We claim that for $P\in(\spec_HR)_{\supseteq\delta(R)}$, $P_x\cap R=P$.

Notice that $P'\subseteq P_x$, and recall $P=P'\cap R$ from Lemma~\ref{lemma:deltaequiv}, whence $P\subseteq P_x\cap R$.

Finally we notice $$R/P\cong S/P_x\cong R/(P_x\cap R).$$  The result follows.
\end{proof}

We recall that at most two ideals in $\spec_HS$ contract to a given ideal in $(\spec_HR)_\delta$.  If $P\in(\spec_HR)_{\supseteq\delta(R)}$, then the two $H$-primes of $S$ which contract to $P$ are $\iota(P)$ and $P_x$, as in Lemmas~\ref{lemma:iota} and \ref{lemma:CGL3}.

We conclude the following:

\begin{obs}Suppose $Q\in(\spec_HS)_{\not\ni x}$ such that $Q$ is not an element of the image of $\iota$.  Then $Q\cap R\in(\spec_HR)_\delta\setminus(\spec_HR)_{\supseteq\delta(R)}.$\end{obs}

\begin{prop}\begin{enumerate}[(a)]
\item For all $\tau$-eigenvectors $r\in R$, $C(r)x^{\drank(r)}\in\langle r\rangle_S.$
\item For all $r\in R$ with $\drank(r)\geq 1$, $C(r)x^{\delta\text{-rank}(r)}\in\langle x\rangle_S.$
\end{enumerate}\label{prop:Crx}\end{prop}

\begin{proof}

If $\drank(r)=0$, then $C(r)=r\in\langle r\rangle_S$.

If $r$ is a $\tau$-eigenvector, then $\tau(r)\in\langle r\rangle_S$.

Notice that $\delta(r)=xr-\tau(r)x\in\langle x\rangle_S$ and is an element of $\langle r\rangle_S$ if $\tau(r)\in\langle r\rangle_S$.  Notice that $\delta^{m}(r)=x\delta^{m-1}(r)-\delta^{m-1}(r)x$ is an element of $\langle x\rangle_S$ for all $m\geq 1$.  An easy induction shows that $\delta^m(r)\in\langle r\rangle_S$ if $\tau(r)\in\langle r\rangle_S$ for all $m$.

We recall from~\ref{thm:Cauchon} that there are nonzero scalars $q_l$ so that $$C(r)=\sum_{l=0}^\infty q_l\delta^l\tau^{-l}(r)x^{-l},$$ whence $$C(r)x^{\drank(r)}=\sum_{l=0}^{\drank(r)} q_l\delta^l\tau^{-l}(r)x^{\drank(r)-l}.$$

If $\drank(r)\geq 1$, it follows that $C(r)x^{\drank(r)}\in\langle x\rangle_S$, as $\delta(R)\subseteq\langle x\rangle_S$.

If $r$ is a $\tau$-eigenvector, say $\tau(r)=\gamma r$ for some $\gamma\in k^\times$, then $$C(r)x^{\drank(r)}=\sum_{l=0}^\infty q_l\gamma^{-l}\delta^l(r)x^{\drank(r)-l}\in\langle r\rangle_S.$$

This establishes both (a) and (b).\end{proof}

To proceed further, we wish for more hypotheses concerning Cauchon extensions.  To be appropriately applicably, any hypothesis added should hold in the cases where $R$ is a CGL extension and where $R$ is a quotient of a CGL extension.

For the next lemma, and any theorems which use it, we also suppose that $\tau$ is acting diagonalizably on $R$, i.e. $R$ is spanned by $\tau$-eigenvectors.

\begin{lemma}Suppose $\tau$ acts diagonalizably on $R$. Then $\tilde{C}$ agrees with $\iota$ on $(\spec_HR)_\delta.$\label{lemma:CGL6}\end{lemma}

\begin{proof}Let $P\in(\spec_HR)_\delta.$  Notice $P$ has a generating set consisting of $\tau$-eigenvectors.

It follows from Proposition~\ref{prop:Crx} that if $P=\langle r_1,\dots,r_n\rangle_R\in(\spec_HR)_\delta$, then $$\langle C(P)\rangle_L=\langle r_1,\dots,r_n\rangle_L,$$ since $r\in\langle C(r),\delta(r)\rangle$ for all $r\in R$.  Lemma~\ref{lemma:iota} tells us $\iota(P)=\langle r_1,\dots,r_n\rangle_S\in\spec_HS$.  And we see $\iota(P)\subseteq \tilde{C}(P)$.

Notice that the extension of $\iota(P)$ to $L$ is $\langle C(P)\rangle_L$; it follows that $\iota(P)=\tilde{C}(P)$, since $\iota(P)$ is prime.\end{proof}

\begin{lemma}Suppose $\tau$ acts diagonalizably on $R$.  If $P\in\spec_HR$, then $\tilde{C}(P)\cap R\subseteq P.$\label{lemma:F(P)capR}
\end{lemma}

\begin{proof}Set $P'=\tilde{C}(P)\cap R$.  Notice $\iota(P')\subseteq \tilde{C}(P)$.  Since $P'\in(\spec_HR)_\delta$, Lemma~\ref{lemma:CGL6} tells us that $\iota(P')=\tilde{C}(P')$.

Hence, $\tilde{C}(P')\subseteq \tilde{C}(P)$.  Since $\tilde{C}$ is an isomorphism, $P'\subseteq P$.\end{proof}

\begin{obs}If $\tau$ acts diagonalizably on $R$, then $\tau$ acts diagonalizably on $S$.\end{obs}

\begin{notation}We fix the following notation for the next several lemmas and observations:
Suppose $\tau$ acts diagonalizably on $R$.

Let $P\in(\spec_HS)_{\not\ni x}$ and $Q\in(\spec_HS)_{\ni x}$.

Set $Q_0=Q\cap R$. Note then that $Q=Q_0+Sx$ and $Q_0\in(\spec_HR)_{\supseteq\delta(R)}.$

Let $\lambda$ denote the inverse of the $h_0$-eigenvalue of $x$, so that $h_0.x=\lambda\inv x.$
\label{notation:difficult}
\end{notation}

\begin{obs}Suppose $a=a_0+a_1x+a_2x^2+\cdots+a_nx^n\in S$ is an $h_0$-eigenvector with eigenvalue $\mu$.  Then $\tau(a_i)=q^{i}\mu a_i$ for all $i$ and $$xa-\mu q^nax=\delta(a_0)+\sum_{i=1}^n[(\mu(\lambda^{i-1}-\lambda^n)a_{i-1}+\delta(a_i)]x^i.$$\label{obs:difficult}
\end{obs}

\begin{lemma}Use the notation of \ref{notation:difficult}.  Let $m>0$ and suppose all elements of $P$ with degree less than $m$ have leading coefficients in $Q_0$.  Then all coefficients of elements of $P$ with degree less than $m$ are in $Q_0$.
\label{lemma:difficult2}
\end{lemma}

\begin{proof} Since $\tau$ acts diagonalizably on $R$, any element of $P$ can be expressed as a sum of $h_0$-eigenvectors in $P$ with the same degrees.  It thus suffices to prove the claim for any $h_0$-eigenvector $a=a_0+a_1x+\cdots +a_nx^n \in P$ of degree $n<m$.  Let $\mu$ be the $h_0$-eigenvalue of $a$.

We proceed by induction on $n$.

Since $P\cap R=0$, the result is trivial if $n=0$.

We assume $n>0$ and induct on $\drank(a_0)$.  If $\drank(a_0)=-1$, then $a_0=0$, so we can write $a=bx$ for some $h_0$-eigenvector $b\in S$ of degree $n-1$.  By induction on $n$, all coefficients of $b$ lie in $Q_0$.  It follows that all $a_i\in Q_0$.

We next assume $\delta$-rank$(a_0)\geq 0$.  Set $c=xa-\mu\lambda^nax$.  Notice that $c$ is an $h_0$-eigenvector in $P$.  By Observation~\ref{obs:difficult}, we see $c$ has degree at most $n$ and constant term of $\drank$ strictly less than $\drank(a_0)$.  By induction, we see all coefficients of $c$ lie in $Q_0$.  Hence, $$\mu(\lambda^{i-1}-\lambda^n)a_{i-1}+\delta(a_i)\in Q_0$$ for $i=1,\dots,n.$  Since $\mu(\lambda^{i-1}-\lambda^n)\neq 0$, and $\delta(a_i)\in Q_0$, we see $a_0,\dots,a_{n-1}\in Q_0$.  We already have $a_n\in Q_0$ by hypothesis, and so the induction is complete.
\end{proof}

\begin{lemma}Use the notation of \ref{notation:difficult}. Assume not all elements of $P$ have leading coefficients in $Q_0$.  Let $m$ be the smallest nonnegative integer such that some element of $P$ with degree $m$ has leading coefficient outside $Q_0$.  Let $d$ be the minimum $\drank$ for constant terms of $h_0$-eigenvectors in $P$ with degree $m$ and leading coefficient not in $Q_0$. Then $d\geq 0$, and any $h_0$-eigenvector $a\in P$ with $a=a_0+a_1x+\cdots+a_mx^m$ such that the $\drank(a_0)<d$ has $a_i\in Q_0$ for all $i$.
\label{lemma:difficult4}
\end{lemma}

\begin{proof}
If $d=-1$, there is an $h_0$-eigenvector $c=c_0+c_1x+\cdots+c_mx^m\in P$ of degree $m$ with $c_m\not\in Q_0$, such that $\drank(c_0)=-1$.  That is, $c_0=0$, and, as in the proof of Lemma~\ref{lemma:difficult2}, we conclude all $c_i\in Q_0$, a contradiction.

Consider now $a$ as in the statement of the lemma.  We wish to show all $a_i\in Q_0$.

We proceed by induction on $\drank(a_0)$.  The case where $\drank(a_0)=-1$ is argued as in the proof of Lemma~\ref{lemma:difficult2}.

So we assume that $0\leq \drank(a_0)<d$.  Notice that $a_m\in Q_0$, by the minimality of $d$.  Let $\mu$ be the $h_0$-eigenvalue of $a$, and consider $c=xa-\mu\lambda^max$, an $h_0$-eigenvector in $P$ with degree at most $m$, and constant term of $\drank$ smaller than that of $a_0$.  By induction, all coefficients of $c$ lie in $Q_0$.  It follows that all $a_0,\dots,a_{m-1}\in Q_0$ as in the proof of Lemma~\ref{lemma:difficult2}.
\end{proof}

\begin{prop}Suppose $\tau$ acts diagonalizably on $R$.

Let $P\in\spec_HS$ and $Q\in(\spec_HS)_{\ni x}.$
\begin{enumerate}[(a)]
\item If $P\cap R=0$, then $P\subseteq Q$
\item Suppose $P\cap R\subseteq Q\cap R$.  Then $P\subseteq Q$.
\end{enumerate}\label{lemma:JcapR}
\end{prop}

\begin{proof}Notice that (b) follows from (a) by passage to the quotient ring $R/Q\cap R$.

We proceed to prove (a).

Notice (a) is clear in the case where $x\in P$ or where $\delta=0$, so we assume $x\not\in P$ and $\delta\neq 0$. Notice that $h_0x=q\inv x$.

Since $x\in Q$, it is sufficient to show that all elements of $P$ have constant term in $Q_0:=Q\cap R$.

If all elements of $P$ have leading coefficients in $Q_0$, the result follows from Lemma~\ref{lemma:difficult2}, so we assume otherwise.  Let $m$ be the smallest nonnegative integer such that some element of $P$ with degree $m$ has leading coefficient outside $Q_0$.  Notice the following observations regarding $m$.

Since $P\cap R=0$, it follows that $m>0$.  There must be some $h_0$-eigenvector of degree $m$ with leading coefficient outside $Q_0$.  From Lemma~\ref{lemma:difficult2}, we observe all coefficients of elements of $P$ with degree less than $m$ lie in $Q_0$.

Choose an $h_0$-eigenvector $a=a_0+a_1x+\cdots+a_mx^m\in P$ with $a_m\neq 0$ and $a_m\not\in Q_0$ such that $a_0$ has minimal $\drank$ $d$.  That is, any element of $P$ of degree $m$ with constant term of $\drank$ less than $d$ has leading coefficient in $Q_0$.  From Lemma~\ref{lemma:difficult4}, we see $d\geq0$.

Let $\mu$ denote the $h_0$-eigenvalue of $a$.  Consider $xa-\mu q^max$, an $h_0$-eigenvector in $P$ with degree at most $m$ and having constant term with $\drank$ strictly less than $d$.  It follows that all coefficients lie in $Q_0$.  As in the proof of Lemma~\ref{lemma:difficult2}, we conclude $a_0,\dots,a_{m-1}\in Q_0$.

We are set up for the main proof.

Let $b=b_0+b_1x+\cdots+b_tx^t\in P$.  We wish to show $b_0\in Q$ and proceed by induction on $t$.  If $t<m$, we are done.  If $t=m$, we consider $ba_m$, noting the leading coefficient is $\mu^mq^{m^2}b_ma_m$, and the constant term is $$\sum_{i=0}^mb_i\delta^i(a_m).$$

Set $c=ba_m-\mu^mq^{m^2}b_ma$, and notice $c\in P$ and the degree of $c$ is less than $m$.  We conclude that all coefficients of $c$ lie in $Q_0$.  In particular, the constant term $$\sum_{i=0}^mb_i\delta^i(a_m)-\mu^mq^{m^2}b_ma_0\in Q_0.$$  Since $\delta^i(a_m)\in Q_0$ for all $i>0$ and $a_0\in Q_0$, we see that $b_0a_m\in Q_0$.  Since $Q_0$ is completely prime, we see $b_0\in Q_0$, as desired.

Now suppose $t>m$.  Consider $$c=ba_m-\mu^tq^{mt}b_tax^{t-m}.$$  Again notice $c\in P$ with degree strictly less than $t$.  By induction, $c_0\in Q_0$.  As above, we conclude $b_0\in Q_0$.
\end{proof}

\begin{notation}Define $\tilde\Psi:\spec_HS\to(\spec_HR)_\delta$ by $P\mapsto P\cap R.$\label{notation:tildePhi}\end{notation}

We collect together relevant facts about $\tilde\Psi$.

\begin{obs}Notice we have established $\tilde\Psi$ is a well-defined surjective homomorphism, as $\tilde\Psi\iota(P)=P$ for all $P\in(\spec_HR)_\delta$ by Lemma~\ref{lemma:iota}.  Moreover $\iota\tilde\Psi(Q)=Q$ for all $Q\in\iota((\spec_HR)_\delta).$

Notice also that $\tilde\Psi$ restricts to an isomorphism $(\spec_HS)_{\ni x}\to(\spec_HR)_{\supseteq\delta(R)}$ by Lemma~\ref{lemma:CGL3}.

We see $\iota\tilde\Psi=\tilde{C}\tilde\Psi$ is a projection.  That $\iota$ agrees with $\tilde{C}$ on the image of $\tilde\Psi$ is Lemma~\ref{lemma:CGL6}.  It is then clear that $\iota\tilde\Psi$ is a projection, in that $(\iota\tilde\Psi)^2=\iota\tilde\Psi.$

A well-known fact is that $\tilde\Psi$ is at most a 2-1 map.  (See for example\cite{BG02} for details.)

For $Q\in\spec_HS$, we notice $\iota\tilde\Psi(Q)\subseteq Q$.

Finally notice that for $P\in\spec_HS$ and $Q\in(\spec_HS)_{\ni x}$, $$P\subseteq Q\Leftrightarrow\tilde\Psi(P)\subseteq\tilde\Psi(Q).$$  This follows from Proposition~\ref{lemma:JcapR}.
\label{obs:tildePsi}
\end{obs}

We recall the notation $W_1,\dots,W_4$ of \ref{notation:BruhatPartition} and the map $\Phi$ from \ref{notation:Phi}.

\begin{thm}Suppose $S=R[x;\tau,\delta]$ is a Cauchon extension with $\tau$ acting diagonalizably on $R$ such that there exists an isomorphism $\nabla:W^{\leq \ov w}\to\spec_HR$, where $W$ is a Coxeter group, $a$ a generator, and $\ov w\in W_a'$.  Suppose $\nabla$ satisfies \begin{enumerate}[(H1)]
\item $\nabla(W_3)=(\spec_HR)_{\supseteq\delta(R)}.$
\item For all $w\in W_1\cup W_2$, $\tilde{C}\nabla(w)\cap R=\nabla\Phi(w).$
\end{enumerate}

Then there is a poset isomorphism $\widetilde\nabla:W^{\leq\ov wa}\to\spec_HS$ given by $$w\mapsto\begin{cases}
\tilde{C}\nabla(w),&w\in W^{\leq\ov w}\\
\nabla(wa)_x,&\text{ else}.
\end{cases}$$

\label{thm:cool}
\end{thm}

In the coming proof, we will name prime ideals after the letter $J$ to avoid confusion with posets.

\begin{proof}
The goal is to demonstrate the hypotheses of Proposition~\ref{prop:IsoExtends} are satisfied by $\P=\tilde{C}(\spec_HR)$, $\tilde\P=\spec_HS$, $\P_x=(\spec_HS)_{\ni x}$ and $\tilde\Phi=\iota\tilde\Psi$, in order to show that $\nabla'=\tilde{C}\nabla$ extends to $\tilde\nabla:W^{\leq\ov wa}\to\spec_HS$.

We have that $\P_x$ is an upper set of $\tilde\P$.  Most of the necessary facts about $\tilde\Phi$ are collected in Observation~\ref{obs:tildePsi}.  We see $$\tilde\Phi:\P_x\to\tilde\Phi(\P_x)=\iota((\spec_HR)_{\supseteq\delta(R)})$$ is an isomorphism such that $\tilde\Phi(J)\subseteq J$ for all $J\in\P_x$. Further, we have that for all $J\in\P_x$ and $J'\in\P$, $$J'\subseteq J\Leftrightarrow\tilde\Phi(J')\subseteq\tilde\Phi(J).$$  Finally, we see $\tilde\Phi$ is a projection onto $\iota((\spec_HR)_{\delta})$.

We see \begin{align*}
\nabla'(W_3)&=\tilde{C}((\spec_HR)_{\supseteq\delta(R)})&\text{by (H1)}\\
&=\iota((\spec_HR)_{\supseteq\delta(R)})&\text{(Lemma~\ref{lemma:CGL6})}\\
&=\tilde\Phi(\P_x).
\end{align*}

As for the final hypothesis of Proposition~\ref{prop:IsoExtends}, by Proposition~\ref{remark:ProjSuffices}, since $\tilde\Phi$ is a projection, we need only check that for all $w\in W_1\cup W_2$, $\tilde\Phi\nabla'(w)=\nabla'\Phi(w)$, which is precisely (H2).

Thus the proposition applies and there exists an isomorphism $\tilde\nabla:W^{\leq\ov wa}\to\spec_HS$ given by $$w\mapsto\begin{cases}
\tilde{C}\nabla(w),&w\in W^{\leq\ov w}\\
\nabla(wa)_x&\text{else}.
\end{cases}$$
\end{proof}

\begin{notation}Set\begin{align*}
\P_3&=(\spec_HR)_{\supseteq\delta(R)}\\
\P_2&=(\spec_HR)_\delta\setminus P_3\\
\P_1&=(\spec_HR)\setminus (P_2\cup P_3)
\end{align*}
\end{notation}

\begin{thm}
Suppose $S=R[x;\tau,\delta]$ is a Cauchon extension with $\tau$ acting diagonalizably on $R$ such that there exists an isomorphism $\nabla:W^{\leq \ov w}\to\spec_HR$, where $W$ is a Coxeter group, $a$ a generator, and $\ov w\in W_a'$.

Suppose that for $i=1,2,3$, $\nabla(W_i)=\P_i$.  Then $\nabla'=\tilde{C}\nabla$ extends to an isomorphism $\tilde\nabla:W^{\leq\ov wa}\to\spec_HS$.\label{thm:main}
\end{thm}

\begin{proof}Assume the hypotheses.

In particular, (H1) of Theorem~\ref{thm:cool} is then satisfied.

We wish to demonstrate (H2), that $\tilde{C}\nabla(w)\cap R=\nabla\Phi(w)$ for all $w\in W_1\cup W_2$.

For $w\in W_2$, this is clear from Lemmas~\ref{lemma:deltaequiv} and \ref{lemma:CGL6}, since $\nabla(w)$ is assumed to be $\delta$-invariant.

Suppose $w\in W_1$.  Hence, $wa\in W_2$ with $wa\leq w$.  It follows that $\nabla(wa)\subseteq\nabla(w)$ is a saturated chain with $\nabla(wa)\in\P_2$ and $\nabla(w)\in\P_1$.  We see from Lemma~\ref{lemma:F(P)capR} that $\tilde C\nabla(w)\cap R\subsetneq\nabla(w)$ since $\nabla(w)\not\in(\spec_HR)_\delta.$  It thus suffices to show that $\nabla(wa)\subseteq\tilde C\nabla(w)\cap R.$

We see $$\iota\nabla(wa)=\tilde{C}\nabla(wa)\subseteq \tilde{C}\nabla(w).$$  Hence, $$\nabla(wa)=\iota\nabla(wa)\cap R\subseteq \tilde{C}\nabla(w)\cap R,$$ as desired.
\end{proof}

\begin{cor}In the notation of Theorems~\ref{thm:cool} and \ref{thm:main}, we see the following diagram commutes:\begin{center}
\begin{tikzpicture}
\node(1) at (0,0){$W^{\leq\ov wa}$};
\node(2) at (3,0){$\spec_HS$};
\node(3) at (0,-2){$W^{\leq\ov w}$};
\node(4) at (3,-2){$\spec_HR$};
\draw[->](1)--node[auto]{$\tilde\nabla$}(2);
\draw[->](1)--node[auto]{$\Phi$}(3);
\draw[->](2)--node[auto]{$\tilde\Psi$}(4);
\draw[->](3)--node[auto]{$\nabla$}(4);
\end{tikzpicture} 

\end{center}\label{cor:projections}
\end{cor}

\section{Applications and examples}\label{section:applications}

Recall the notation of Theorem~\ref{thm:main}, which will be used heavily throughout this section.

\begin{prop}Let $W$ be a Coxeter group, $a$ a generator, and choose $\ov w\in W$ with $a\not\in W^{\leq\ov w}$.  Suppose $W^{\leq\ov w}\cong\spec_HR$, and suppose $S=R[x;\tau]$ is a Cauchon extension.

Then $W^{\leq\ov wa}\cong\spec_HS$.\label{prop:delta0}\end{prop}

\begin{proof}

Notice $a\not\in W^{\leq\ov w}$ is equivalent to saying $\ov w\in W_a'$.

Suppose $\nabla:W^{\leq\ov w}\to\spec_HR$ is an isomorphism.  Notice that $W_1=W_2=\emptyset$, so hypothesis (H2) of Theorem~\ref{thm:cool} is satisfied trivially.  Thus $W^{\leq\ov w}=W_3$; meanwhile, $\spec_HR=(\spec_HR)_{\supseteq\delta}$, so hypothesis (H1) is satisfied.  Hence, $\nabla$ extends to an isomorphism $\tilde\nabla:W^{\leq\ov wa}\to\spec_HR$ by Theorem~\ref{thm:cool}.
\end{proof}

Notice in Proposition~\ref{prop:delta0}, that $\tilde\nabla(a)=\langle x\rangle_S$ and $\tilde\nabla(wa)=\nabla(w)_x$ for any $w\in W^{\leq\ov w}.$\label{remark:delta0}

The following corollary is a well-known result, but notice that it is an immediate consequence of Proposition~\ref{prop:delta0}, which in turn can follow from Theorem~\ref{thm:main}.

\begin{eg}Let $W$ be a Coxeter group with generating set $\{s_1,\dots,s_n\}$.  Let $c=s_1\dots s_n$ be a Coxeter element.  Then $W^{\leq c}\cong\spec_H\O_{\mathbf q}(k^n)$ for any multiplicatively antisymmetric matrix $\mathbf q$.\label{cor:quantumaffinespec}\end{eg}

Recall from Observation~\ref{obs:BruhatPartitionfacts} that $W_1=m_a(W_2)$.

\begin{eg}We demonstrate Theorem~\ref{thm:main} on the multiparameter $2\times 2$ quantum matrix algebra.

Let $W$ be the Coxeter group of type $A_3$.

 Set $R=\O_{\q}(k^3)$ for some multiplicatively antisymmetric matrix $\q=(q_{ij})$.   Suppose $q$ is not a root of unity and let $p_2,p_3, \gamma\in k^\times$.  Let $S=R[x;\tau,\delta]$ be a Cauchon extension where $\delta(x_1)=\gamma x_2x_3$, $\tau(x_1)=qp_1p_2x_1$ and for $i=2,3$, we have $\tau(x_i)=p_ix_i$ and $\delta(x_i)=0$.    Set $a=s_2$ and $\ov w=s_2s_1s_3$.  We have that $W^{\leq\ov w}\cong\spec_HR$ via a canonical map $\nabla$, as implicitly defined by Corollary~\ref{cor:quantumaffinespec}.  Here $W_2=\{1\}$, where $1$ denotes the identity element of $W$.  $W_1=\{s_2\}$.  $W_3$ consists of the remaining six words which can be seen as the upper set generated by $\{s_1,s_3\}$.

Notice that $(\spec_HR)_{\supseteq\delta(R)}$ is the upper set generated by the ideals $\langle x_2\rangle$ and $\langle x_3\rangle$.  Notice that $\nabla(s_1)=\langle x_2\rangle$ and $\nabla(s_3)= \langle x_3\rangle$.  We thus see that $\nabla(W_3)=\P_3$.

We next notice that $\nabla(1)=0$ is $\delta$-invariant, and $\nabla(s_2)=\langle x_1\rangle$ is not, so $\nabla(W_2)=\P_2$.  It follows that $\nabla(W_1)=\P_1$.

We thus conclude there is a poset isomorphism $W^{\ov wa}\to\spec_HS.$
\label{eg:app2qmatrices}\end{eg}

\begin{dfn}
Let $Q=(q_1,\dots,q_n)$ be a vector in $(k^\times)^n$, and let $\Gamma=(\gamma_{ij})$ be a multiplicatively antisymmetric matrix over $k$.  The multiparameter quantized Weyl algebra $Y_n=Y_n^{Q,\Gamma}(k)$ is presented by generators $x_n,\dots,x_1,y_1,\dots,y_n$ and relations \begin{align*}
y_iy_j&=\gamma_{ij}y_jy_i&\forall\ i,j\\
x_iy_j&=\gamma_{ji}y_jx_i&(i<j)\\
x_iy_j&=q_j\gamma_{ji}y_jx_i&(i>j)\\
x_ix_j&=q_i\gamma_{ij}x_jx_i&(i<j)\\
x_iy_i&=q_iy_ix_i+1+\sum_{l<i}(q_l-1)y_lx_l&\forall\ i
\end{align*}
\end{dfn}

\begin{prop}
Let $Y_n$ be as above.  Let $W$ be the Coxeter group of type $A_n$.  Set $w_n=s_ns_{n-1}\cdots s_2s_1s_2\cdots s_n.$  Then $W^{\leq w_n}\cong\spec_HY_n$.
\label{eg:app2weylalg}
\end{prop}

 \begin{proof}

 Notice $Y_n\cong Y_{n-1}[x_n;\sigma_n][y_n;\tau_n,\delta_n]$ where
\begin{align*}\sigma_n&:\begin{cases}
 x_i\mapsto q_i\inv\gamma_{ni}x_i,&(i<n)\\
 y_i\mapsto q_i\gamma_{in}y_i,&(i<n),
 \end{cases}\\
  \tau_n&:\begin{cases}
 x_i\mapsto \gamma_{in}x_i,&(i<n)\\
 x_n\mapsto q_n\inv x_n\\
 y_i\mapsto \gamma_{ni}y_i,&(i<n),
 \end{cases}\\ \delta_n&:\begin{cases}
 y_i\mapsto0,&(i<n)\\
 x_i\mapsto0,&(i<n)\\
 x_n\mapsto-q_n\inv\left(1+\sum_{l<n}(q_l-1)y_lx_l\right).
 \end{cases}\end{align*}

We set $R=Y_{n-1}[x_n;\sigma_n]$ with the variables reordered so that $x_n$ is first (noticing we still have a CGL extension under this ordering of the variables), and write $Y_n=R[y_n;\tau_n,\delta_n]$.  Set $\Omega_n=C(x_n)y_n$.  We observe that $\langle\Omega_n\rangle\in\spec_HY_n$.

We proceed to show by induction on $n$ that there is an isomorphism $$\nabla_n:W^{\leq w_n}\to\spec_HY_n$$ such that $s_n\mapsto\langle\Omega_n\rangle$ and $\nabla_n(w)\subseteq Y_{n-1}$ for all $w\leq w_{n-1}$.

Notice $w_1=s_1$, and the claim for $n=1$ is a well-known result (if we define $Y_0=k$ $w_0=1$).  See \cite{G92}

Suppose inductively we have constructed isomorphisms $\nabla_l=\widetilde\nabla_{l-1}:W^{\leq w_l}\to \spec_HY_l$ as in Theorem~\ref{thm:cool} for $2\leq l\leq n-1$ such that $\nabla_l(s_l)=\langle\Omega_l\rangle$ and $\nabla_l(w)\subseteq Y_{l-1}$ for all $w\leq w_{l-1}$.

We see from Proposition~\ref{prop:delta0} that $\spec_HR$ lines up with the Bruhat interval corresponding to the word $w_{n-1}s_n$.  It is not difficult to see, that if we order the variables with $x_n$ first, and set $\ov w=s_nw_{n-1}$, then there is a poset isomorphism $\nabla:W^{\leq\ov w}\to\spec_HR$ such that $$w\mapsto\nabla_{n-1}(w)\text{ and }s_nw\mapsto\nabla_{n-1}(w)_{x_n}\text{ for all }w\leq w_{n-1}.$$

Then $(\spec_HR)_{\supseteq\delta(R)}$ is the upper set generated by $\langle\Omega_{n-1}\rangle$.  $W_3$ is the upper set generated by $s_{n-1}$, and $\nabla(s_{n-1})=\langle\Omega_{n-1}\rangle$.  We thus see that $\nabla(W_3)=\P_3$.

We see $W_2=W^{\leq w_{n-2}}$ and $W_1=\{s_nw:w\in W_2\}$.  We notice that $\nabla(s_n)=\langle x_n\rangle$.  Notice $\nabla(W_2)\subseteq\spec_H Y_{n-1}$ and that $\delta(Y_{n-1})=0$.  We conclude $\nabla(W_2)\subseteq\P_2$.

Next notice that $\nabla(s_nw)=\nabla_{n-1}(w)_{x_n}$ for all $w\leq w_{n-2}$.  Thus $\Omega_{n-1}\in\delta\nabla(s_nw)$.  But $\Omega_{n-1}\not\in\nabla(s_nw)$.  We conclude $\nabla(W_i)=\P_i$ for $i=1,2,3.$

We then see we have an isomorphism $\tilde\nabla:W^{\leq w_n}\to\spec_HY_n$ as in Theorem~\ref{thm:main}.

We then see that $$\widetilde\nabla(s_n)=\langle C(x_n)\rangle_L\cap Y_n=\langle\Omega_n\rangle_L\cap Y_n.$$

This equals $\langle\Omega_n\rangle_S$ since $\langle\Omega_n\rangle_S$ is prime.

Set $\nabla_n=\widetilde\nabla.$  We have then established that $\nabla_n(s_n)=\langle\Omega_n\rangle.$

Suppose $w\leq w_{n-1}$.  Then $\nabla_n(w)=\tilde{C}\nabla(w)=\iota\nabla(w)$, since $\nabla(w)\in Y_{n-1}$ and $\delta(Y_{n-1})=0$.  This completes the induction.
\end{proof}

See Figure~\ref{figure:32123} for an illustration of the $n=3$ case.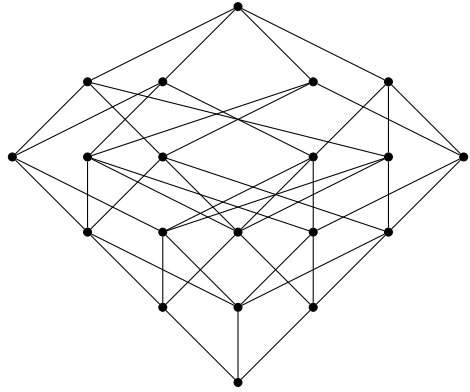
\begin{figure}
\begin{center}
\begin{tikzpicture}
[vertex/.style={circle,draw,fill,inner sep=0 pt, minimum size=3},scale=.5]
\node(empty) at (0,0)[vertex]{};

\node(1) at (-2,2)[vertex]{};
\node(2) at (0,2)[vertex]{};
\node(0) at (2,2)[vertex]{};

\node(12) at (-4,4)[vertex]{};
\node(21) at (-2,4)[vertex]{};
\node(10) at (0,4)[vertex]{};
\node(20) at (2,4)[vertex]{};
\node(02) at (4,4)[vertex]{};

\node(212) at (-6,6)[vertex]{};
\node(012) at (-2,6)[vertex]{};
\node(120) at (-4,6)[vertex]{};
\node(210) at (2,6)[vertex]{};
\node(021) at (4,6)[vertex]{};
\node(020) at (6,6)[vertex]{};

\node(0212) at (-4,8)[vertex]{};
\node(2120) at (-2,8)[vertex]{};
\node(0210) at (4,8)[vertex]{};
\node(0120) at (2,8)[vertex]{};

\node(02120) at (0,10)[vertex]{};

\draw(empty)--(1);
\draw(empty)--(2);
\draw(empty)--(0);

\draw(1)--(12);
\draw(1)--(21);
\draw(1)--(10);
\draw(2)--(20);
\draw(2)--(02);
\draw(2)--(21);
\draw(2)--(12);
\draw(0)--(02);
\draw(0)--(10);
\draw(0)--(20);

\draw(10)--(012);
\draw(10)--(021);
\draw(10)--(120);
\draw(10)--(210);

\draw(12)--(120);
\draw(12)--(212);
\draw(12)--(012);
\draw(21)--(210);
\draw(21)--(212);
\draw(21)--(021);
\draw(20)--(020);
\draw(20)--(210);
\draw(20)--(120);
\draw(02)--(020);
\draw(02)--(021);
\draw(02)--(012);

\draw(120)--(2120);
\draw(120)--(0120);
\draw(210)--(2120);
\draw(210)--(0210);
\draw(212)--(2120);
\draw(212)--(0212);
\draw(012)--(0212);
\draw(012)--(0120);
\draw(021)--(0212);
\draw(020)--(0210);
\draw(020)--(0120);
\draw(021)--(0210);

\draw(0212)--(02120);
\draw(0120)--(02120);
\draw(0210)--(02120);
\draw(2120)--(02120);

\end{tikzpicture}
\end{center} 
\caption{The interval $[1,s_3s_2s_1s_2s_3]$ in the Coxeter group of type $A_3$, isomorphic to the invariant prime spectrum of the quantized Weyl algebra in 6 variables.}\label{figure:32123}
\end{figure}

\begin{cor}
$Y_n$ is not a De Concini-Kac-Procesi algebra.
\end{cor}

\begin{proof}
Let $U_q^w$ be a De Concini-Kac-Procesi algebra for a word $w$ in a Weyl group $W$.  Let $m$ be the length of $w$.  Then $U_q^w$ is a CGL extension in $m$ variables, and so has Gelfand-Kirilov (GK) dimension $m$.

It follows from Theorem~\ref{thm:Yaki} that the height of $\spec_H(U_q^w)$ is also $m$.  We recall from Proposition~\ref{prop:independt2} that $\spec_H(U_q^w)$ is an invariant of the filtered algebra.

That is, every De Concini-Kac-Procesi algebra has a torus-invariant prime spectrum with height equal to its GK-dimension.

$Y_n$ is a CGL extension in $2n$ variables, and so has GK-dimension $2n$; however, the height of $\spec_H(Y_n)$ is $2n-1$.
\end{proof}

The Horton algebra was first described by Horton in \cite{H01} to capture in one family the algebras quantum euclidean space and quantum symplectic space.  Quantum euclidean space is known to be isomorphic to a De Concini-Kac-Procesi algebra, but quantum symplectic space is not.  Moreover, quantum symplectic space is not isomorphic to a 2-cocycle twist of quantum euclidean space.

\begin{dfn}
Let $P,Q\in (k^\times)^n$, $P=(p_1,\dots,p_n)$
and $Q=(q_1,\dots,q_n)$ such that $p_iq_i\inv\neq\sqrt[\bullet]1$ for all $i$.
Suppose $\Gamma=(\gamma_{ij})\in M_n(k^\times)$ is a multiplicatively
antisymmetric matrix.
The Horton algebra $K_n=K_{n,\Gamma}^{P,Q}(k)$ is an
algebra presented by generators $x_n,\dots,x_1,y_1,\dots,y_n$ and
relations
\begin{align*}
y_iy_j&=\gamma_{ij}y_jy_i&\forall\ i,j\\
x_iy_j&=p_j\gamma_{ji}y_jx_i&(i<j)\\
x_iy_j&=q_j\gamma_{ji}y_jx_i&(i>j)\\
x_ix_j&=q_ip_j\inv\gamma_{ij}x_jx_i&(i<j)\\
x_iy_i&=q_iy_ix_i+\sum_{l<i}(q_l-p_l)y_lx_l&\forall\ i
\end{align*}
\end{dfn}
The torus-invariant prime spectrum of quantum euclidean space, denoted $\O_q(\mathfrak o(k^{2n})$ was described by Oh and Park in \cite{OP98}.  Their work was generalized to the family $K_n$ by Horton \cite{H01}.  In particular, it follows from Horton's work that all members of the family have isomorphic invariant prime spectra.  Quantum euclidean space was realized as a De Concini-Kac-Procesi algebra for $w_n$ by Kamita\cite{K00}.  It then follows from Yakimov \cite{Y09} that $\spec_H(\O_q(\mathfrak o(k^{2n})))\cong W^{\leq w_n}$, and hence that $\spec_HK_n\cong W^{\leq w_n}$.

However, this is also an easy application of Theorem~\ref{thm:main} in a manner identical  to the proof of Proposition~\ref{eg:app2weylalg}.

\begin{prop}Let $K_n$ be as above and $W$ be the Weyl group of type $D_{n+1}$.

Set $w_l=s_l\dots s_2s_1s_0s_2\cdots s_l$ for $l=2\dots n$.

Then $\spec_HK_n\cong W^{\leq w_n}$ for $n\geq 2$.  See Figure~\ref{figure:K3}.\label{prop:specKn}\end{prop}\begin{figure}[h]
\begin{tikzpicture}
[vertex/.style={circle,draw,fill,inner sep=0 pt, minimum size=3},scale=.5]
\node(empty) at (0,0)[vertex]{};

\node(1) at (-3,2)[vertex]{};
\node(0) at (-1,2)[vertex]{};
\node(2) at (1,2)[vertex]{};
\node(3) at (3,2)[vertex]{};

\node(12) at (-6,4)[vertex]{};
\node(21) at (-4.5,4)[vertex]{};
\node(10) at (-3,4)[vertex]{};
\node(20) at (-1.5,4)[vertex]{};
\node(13) at (0,4)[vertex]{};
\node(02) at (1.5,4)[vertex]{};
\node(30) at (3,4)[vertex]{};
\node(32) at (4.5,4)[vertex]{};
\node(23) at (6,4)[vertex]{};

\node(212) at (-7,6)[vertex]{};
\node(012) at (-6,6)[vertex]{};
\node(120) at (-5,6)[vertex]{};
\node(130) at (-2,6)[vertex]{};
\node(021) at (-4,6)[vertex]{};
\node(210) at (-3,6)[vertex]{};
\node(132) at (-1,6)[vertex]{};
\node(020) at (1,6)[vertex]{};
\node(213) at (2,6)[vertex]{};
\node(230) at (3,6)[vertex]{};
\node(023) at (4,6)[vertex]{};
\node(320) at (5,6)[vertex]{};
\node(032) at (6,6)[vertex]{};
\node(232) at (7,6)[vertex]{};

\node(0212) at (-6,8)[vertex]{};
\node(2120) at (-5,8)[vertex]{};
\node(0210) at (-4,8)[vertex]{};
\node(0120) at (-3,8)[vertex]{};
\node(1320) at (-2,8)[vertex]{};
\node(2130) at (-1,8)[vertex]{};
\node(0132) at (0,8)[vertex]{};
\node(2132) at (1,8)[vertex]{};
\node(0213) at (2,8)[vertex]{};
\node(0320) at (3,8)[vertex]{};
\node(0230) at (4,8)[vertex]{};
\node(2320) at (5,8)[vertex]{};
\node(0232) at (6,8)[vertex]{};

\node(02120) at (-5,10)[vertex]{};
\node(21320) at (-3,10)[vertex]{};
\node(02130) at (-1,10)[vertex]{};
\node(01320) at (1,10)[vertex]{};
\node(02132) at (3,10)[vertex]{};
\node(02320) at (5,10)[vertex]{};

\node(021320) at (0,12)[vertex]{};

\draw(empty)--(1);
\draw(empty)--(2);
\draw(empty)--(3);
\draw(1)--(12);
\draw(1)--(13);
\draw(1)--(21);
\draw(2)--(21);
\draw(2)--(12);
\draw(2)--(23);
\draw(2)--(32);
\draw(3)--(13);
\draw(3)--(23);
\draw(3)--(32);
\draw(12)--(212);
\draw(12)--(132);
\draw(21)--(212);
\draw(21)--(213);
\draw(13)--(132);
\draw(13)--(213);
\draw(23)--(213);
\draw(23)--(232);
\draw(32)--(232);
\draw(32)--(132);
\draw(212)--(2132);
\draw(132)--(2132);
\draw(213)--(2132);
\draw(232)--(2132);
\draw(0)--(10);
\draw(0)--(20);
\draw(0)--(30);
\draw(10)--(120);
\draw(10)--(130);
\draw(10)--(210);
\draw(20)--(210);
\draw(20)--(120);
\draw(20)--(230);
\draw(20)--(320);
\draw(30)--(130);
\draw(30)--(230);
\draw(30)--(320);
\draw(120)--(2120);
\draw(120)--(1320);
\draw(210)--(2120);
\draw(210)--(2130);
\draw(130)--(1320);
\draw(130)--(2130);
\draw(230)--(2130);
\draw(230)--(2320);
\draw(320)--(2320);
\draw(320)--(1320);
\draw(2120)--(21320);
\draw(1320)--(21320);
\draw(2130)--(21320);
\draw(2320)--(21320);
\draw(empty)--(0);
\draw(1)--(10);
\draw(2)--(20);
\draw(3)--(30);
\draw(21)--(21);
\draw(12)--(12);
\draw(23)--(23);
\draw(32)--(32);
\draw(13)--(13);
\draw(213)--(213);
\draw(212)--(2120);
\draw(232)--(2320);
\draw(132)--(1320);
\draw(2132)--(21320);

\draw(21)--(210);
\draw(32)--(320);
\draw(12)--(120);

\draw(02)--(021);
\draw(02)--(012);
\draw(02)--(023);
\draw(02)--(032);
\draw(012)--(0212);
\draw(012)--(0132);
\draw(021)--(0212);
\draw(021)--(0213);
\draw(023)--(0213);
\draw(023)--(0232);
\draw(032)--(0232);
\draw(032)--(0132);
\draw(0212)--(02132);
\draw(0132)--(02132);
\draw(0213)--(02132);
\draw(0232)--(02132);
\draw(020)--(0210);
\draw(020)--(0120);
\draw(020)--(0230);
\draw(020)--(0320);
\draw(0120)--(02120);
\draw(0120)--(01320);
\draw(0210)--(02120);
\draw(0210)--(02130);
\draw(0230)--(02130);
\draw(0230)--(02320);
\draw(0320)--(02320);
\draw(0320)--(01320);
\draw(02120)--(021320);
\draw(01320)--(021320);
\draw(02130)--(021320);
\draw(02320)--(021320);
\draw(02)--(020);
\draw(021)--(021);
\draw(012)--(012);
\draw(023)--(023);
\draw(032)--(032);
\draw(0213)--(0213);
\draw(0212)--(02120);
\draw(0232)--(02320);
\draw(0132)--(01320);
\draw(02132)--(021320);

\draw(02132)--(2132);
\draw(02320)--(2320);
\draw(01320)--(1320);
\draw(021320)--(21320);
\draw(0212)--(212);
\draw(0213)--(213);
\draw(0232)--(232);
\draw(0320)--(320);
\draw(0230)--(230);
\draw(02130)--(2130);
\draw(02120)--(2120);
\draw(012)--(12);
\draw(02)--(2);
\draw(032)--(32);
\draw(0210)--(210);
\draw(0120)--(120);
\draw(020)--(20);
\draw(021)--(21);
\draw(023)--(23);

\draw(0)--(02);
\draw(10)--(012);
\draw(10)--(021);
\draw(30)--(032);
\draw(30)--(023);
\draw(130)--(0213);
\draw(130)--(0132);

\end{tikzpicture} \caption{$\spec_HK_n$ when $n=3$, which is isomorphic to the Bruhat order interval $[1,s_3s_2s_1s_0s_2s_3]$ in the Coxeter group of type $D_4$, as in Proposition~\ref{prop:specKn}}\label{figure:K3}
\end{figure}
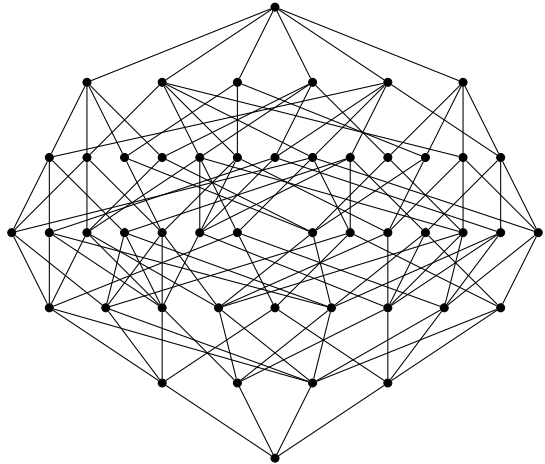

\begin{eg}Set $R=\O_q(M_2(k))$.  The canonical generators of $\O_q(M_2(k))$ are $X_{11},X_{12},X_{21},X_{22}$.  To match CGL notation, we relabel these $x_1,\dots,x_4$ respectively.  We label the quantum determinant $D_q=x_1x_4-qx_2x_3.$

Set $\ov w=s_2s_1s_3s_2$, an element of $W$, the Coxeter group of type $A_3$.  Set $a=s_1$.  We have from Example~\ref{eg:app2qmatrices} a poset isomorphism $\nabla:W^{\leq\ov w}\to\O_q(M_2)$.

Notice $W_3$ is generated as an upper set in $W^{\leq\ov w}$ by $s_3s_2$.  We also have $$W_2=\{1,s_2,s_3,s_2s_3,s_1s_2\}\text{ and }W_1=\{s_1,s_2s_1,s_1s_3,s_2s_1s_3,s_2s_1s_2\}.$$  Let $\gamma\in k^\times$.

There exist an automorphism $\tau$ and $\tau$-derivation $\delta$ on $\O_q(M_2(k))$ such that $$\tau:\begin{cases}
x_1\mapsto qx_1\\
x_2\mapsto qx_2\\
x_3\mapsto q\inv x_3\\
x_4\mapsto q\inv x_4
\end{cases}\text{ and }\delta:\begin{cases}x_1\mapsto \gamma x_3\\
x_2\mapsto \gamma x_4\\
x_3\mapsto 0\\
x_4\mapsto 0
\end{cases}.$$

Thus we have a skew polynomial algebra $S=\O_q(M_2)[x_5;\tau,\delta]$.

There is an action by a torus $H=(k^\times)^3$ such that for $h=(\alpha_1,\alpha_2,\alpha_3)\in H$, $$h.x_i=\alpha_ix_i,\ \text{ for }i=1,\dots,3,\ h.x_4=\alpha_1\inv\alpha_2\alpha_3x_4,\ h.x_5=\alpha_1\inv\alpha_3x_5.$$

$S$ is a Cauchon extension with respect to this torus.

Notice that $D_q$ is central in $S$ and so $\delta(D_q)=0$.

Notice $\delta(R)\subseteq\langle x_3,x_4\rangle=\nabla(s_3s_2)$ and that $s_3s_2$ generates $W_3$.  One easily checks that $\nabla(W_i)=\P_i$ for $i=1,\dots,3$, and thus $W^{\leq\ov wa}\cong\spec_HS$.

See Figure~\ref{figure:21321-A3}.\begin{figure}\begin{tikzpicture}
[vertex/.style={circle,draw,fill,inner sep=0 pt, minimum size=3},scale=.5]
\node(empty) at (0,0)[vertex]{};
\node(1) at (-3,2)[vertex]{};
\node(2) at (0,2)[vertex]{};
\node(3) at (3,2)[vertex]{};
\node(12) at (-2,4)[vertex]{};
\node(21) at (-4,4)[vertex]{};
\node(13) at (0,4)[vertex]{};
\node(32) at (4,4)[vertex]{};
\node(23) at (2,4)[vertex]{};
\node(212) at (-4,6)[vertex]{};
\node(213) at (-2,6)[vertex]{};
\node(132) at (2,6)[vertex]{};
\node(232) at (4,6)[vertex]{};
\node(2132) at (0,8)[vertex]{};
\node(321) at (0,6)[vertex]{};
\node(1321) at (-3,8)[vertex]{};
\node(2321) at (3,8)[vertex]{};
\node(21321) at (0,10)[vertex]{};
\draw(empty)--(1);
\draw(empty)--(2);
\draw(empty)--(3);
\draw(1)--(12);
\draw(1)--(13);
\draw(1)--(21);
\draw(2)--(21);
\draw(2)--(12);
\draw(2)--(23);
\draw(2)--(32);
\draw(3)--(13);
\draw(3)--(23);
\draw(3)--(32);
\draw(12)--(212);
\draw(12)--(132);
\draw(21)--(212);
\draw(21)--(213);
\draw(13)--(132);
\draw(13)--(213);
\draw(23)--(213);
\draw(23)--(232);
\draw(32)--(232);
\draw(32)--(132);
\draw(212)--(2132);
\draw(132)--(2132);
\draw(213)--(2132);
\draw(232)--(2132);
\draw(321)--(2321);
\draw(321)--(1321);
\draw(1321)--(21321);
\draw(2321)--(21321);
\draw(32)--(321);
\draw(132)--(1321);
\draw(232)--(2321);
\draw(2132)--(21321);
\draw(13)--(321);
\draw(21)--(321);
\draw(212)--(1321);
\draw(213)--(2321);
\end{tikzpicture} \caption{The Bruhat order interval $[1,s_2s_1s_3s_2s_1]$ in the Weyl group of type $A_3$, isomorphic to the invariant prime spectrum of $S=\O_q(M_2)[x;\tau,\delta]$ as in Example~\ref{eg:21321-A3}}\label{figure:21321-A3}\end{figure}
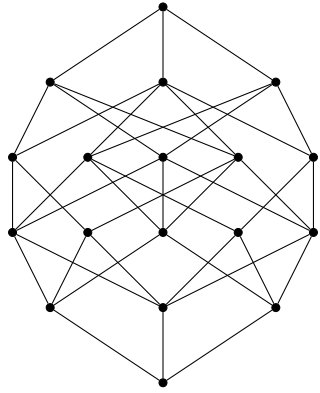
\label{eg:21321-A3}
\end{eg}

\begin{eg}
Set $R=\O_q(M_2(k))$ with generators as in Example~\ref{eg:21321-A3}.  Set $\ov w=s_2s_1s_0s_2$, an element of $W$, the Coxeter group of type $\tilde A_2$, with generators as pictured. 
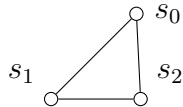
\begin{center}\begin{tikzpicture}
[vertex/.style={circle,draw,inner sep=0 pt, minimum size=5}]
\node(1) [vertex][label=above left:$s_1$]{};
\node(2) [vertex][right=of 1,label=above right:$s_2$]{};
\node(0) [vertex] [above right=of 1,label=right:$s_0$]{};


\
\draw(1)--(2);
\draw(0)--(2);
\draw(0)--(1);
\end{tikzpicture} 
\end{center}  Set $a=s_1$.  (Compare to Example~\ref{eg:21321-A3}.  The only difference is the group in which the string of letters is being viewed.)  We have established a poset isomorphism $\nabla:W^{\leq\ov w}\to\O_q(M_2)$.

Notice $W_3$ is generated as in upper set in $W^{\leq\ov w}$ by $s_0$.  We also have $$W_2=\{1,s_2,s_1s_2\}\text{ and }W_1=\{s_1,s_2s_1,s_2s_1s_2\}.$$  Let $\gamma\in k^\times$.

There exist an automorphism $\tau$ and $\tau$-derivation $\delta$ on $\O_q(M_2(k))$ such that $$\tau:\begin{cases}
x_1\mapsto x_1\\
x_2\mapsto qx_2\\
x_3\mapsto q\inv x_3\\
x_4\mapsto x_4
\end{cases}\text{ and }\delta:\begin{cases}x_1\mapsto \gamma x_3^2\\
x_2\mapsto \gamma x_3x_4\\
x_3\mapsto 0\\
x_4\mapsto 0
\end{cases}.$$

Thus we have a skew polynomial algebra $S=\O_q(M_2)[x_5;\tau,\delta]$.

There is an action by a torus $H=(k^\times)^3$ such that for $h=(\alpha_1,\alpha_2,\alpha_3)\in H$, $$h.x_i=\alpha_ix_i,\ \text{ for }i=1,\dots,3,\ h.x_4=\alpha_1\inv\alpha_2\alpha_3x_4,\ h.x_5=\alpha_1\inv\alpha_3^2x_5.$$

$S$ is a Cauchon extension with respect to $H$.

Notice that $D_q$ is central in $S$ and so $\delta(D_q)=0$.

Notice $\delta(R)\subseteq\langle x_3\rangle=\nabla(s_0)$ and that $s_0$ generates $W_3$.  One easily checks that $\nabla(W_i)=\P_i$ for $i=1,\dots,3$, and thus $W^{\leq\ov w}\cong\spec_HS$.

See Figure~\ref{figure:21321-tildeA2}.\begin{figure}\begin{tikzpicture}
[vertex/.style={circle,draw,fill,inner sep=0 pt, minimum size=3},scale=.5]
\node(empty) at (0,0)[vertex]{};
\node(1) at (-3,2)[vertex]{};
\node(2) at (0,2)[vertex]{};
\node(3) at (3,2)[vertex]{};
\node(12) at (-3,4)[vertex]{};
\node(21) at (-5,4)[vertex]{};
\node(13) at (1,4)[vertex]{};
\node(32) at (5,4)[vertex]{};
\node(23) at (3,4)[vertex]{};
\node(212) at (-5,6)[vertex]{};
\node(213) at (-1,6)[vertex]{};
\node(132) at (0,6)[vertex]{};
\node(232) at (3,6)[vertex]{};
\node(2132) at (2,8)[vertex]{};
\node(31) at (-1,4)[vertex]{};
\node(131) at (1,6)[vertex]{};
\node(321) at (5,6)[vertex]{};
\node(231) at (-3,6)[vertex]{};
\node(2131) at (-2,8)[vertex]{};
\node(1321) at (-4,8)[vertex]{};
\node(2321) at (4,8)[vertex]{};
\node(21321) at (0,10)[vertex]{};

\draw(1)--(21);
\draw(21)--(212);

\draw(1)--(12);

\draw(1)--(13);
\draw(21)--(213);
\draw(212)--(2132);

\draw(empty)--(2);
\draw(2)--(12);

\draw(empty)--(1);
\draw(2)--(21);
\draw(12)--(212);

\draw(empty)--(3);
\draw(2)--(23);
\draw(2)--(32);
\draw(12)--(132);

\draw(3)--(13);
\draw(23)--(213);
\draw(32)--(232);
\draw(32)--(132);
\draw(132)--(2132);
\draw(232)--(2132);
\draw(13)--(213);
\draw(3)--(23);
\draw(3)--(32);
\draw(23)--(232);
\draw(13)--(132);
\draw(213)--(2132);

\draw(31)--(131);
\draw(231)--(2131);
\draw(321)--(2321);
\draw(321)--(1321);
\draw(1321)--(21321);
\draw(2321)--(21321);
\draw(131)--(2131);
\draw(31)--(231);
\draw(31)--(321);
\draw(231)--(2321);
\draw(131)--(1321);
\draw(2131)--(21321);

\draw(3)--(31);
\draw(13)--(131);
\draw(32)--(321);
\draw(23)--(231);
\draw(213)--(2131);
\draw(132)--(1321);
\draw(232)--(2321);
\draw(2132)--(21321);

\draw(1)--(31);
\draw(21)--(321);
\draw(21)--(231);
\draw(212)--(1321);

\end{tikzpicture} \caption{The Bruhat order interval $[1,s_2s_1s_0s_2s_1]$ in the affine Weyl group of type $\tilde A_2$, isomorphic to the invariant prime spectrum of $S=\O_q(M_2)[x;\tau,\delta]$ as in Example~\ref{eg:21321-tildeA2}}\label{figure:21321-tildeA2}\end{figure}
\label{eg:21321-tildeA2}
\end{eg}

\bibliographystyle{abbrv}
\bibliography{refs}

\end{document}